\documentclass[11pt,reqno]{amsart}

\usepackage{amssymb,amsmath,amsthm,latexsym,color}
\usepackage{pifont}
\usepackage{verbatim}
\usepackage{amsfonts}

\title[Nonlinear Heat
Equations Involving Highly Singular Initial Values]{The Nonlinear Heat Equation involving Highly Singular Initial Values and new blowup and life span results}
\author[S. Tayachi and F. B. Weissler] {Slim Tayachi$^1$ and Fred B. Weissler$^2$}
\font\TenEns=msbm10 \font\SevenEns=msbm7 \font\FiveEns=msbm5
\newfam\Ensfam

\textfont\Ensfam=\TenEns \scriptfont\Ensfam=\SevenEns
\scriptscriptfont\Ensfam=\FiveEns

\def\R {\mathbb R}

\def\Rd {{\mathbb R}^N}

\def\Dd {{\mathcal D'}(\Omega)}
\def\Sd {{\mathcal S'}(\Rd)}
\def\D {{\mathcal D}(\Omega_m)}
\def\X {{\mathcal X}}

\newtheorem{To}{Theorem}[section]
\newtheorem{prop}[To]{Proposition}
\newtheorem{cor}[To]{Corollary}

\newtheorem{rem}[To]{Remark}
\newtheorem{defini}[To]{Definition}

\newtheorem{rems}[To]{Remarks}
\date{\today}
\begin{document}

\maketitle

\begin{center}{$^1$
Universit\'e de Tunis El Manar, Facult\'e des Sciences de Tunis, 
\\ D\'epartement de
Math\'ematiques, 
\\ Laboratoire  \'Equations aux D\'eriv\'ees
Partielles LR03ES04, 
\\ 2092 Tunis,
Tunisie.
\\ e-mail: {\tt slim.tayachi@fst.rnu.tn} \vspace{1cm}
\\
$^2$ Universit\'e Paris 13, Sorbonne Paris Cit\'e,
\\ CNRS UMR 7539 LAGA,
 \\ 99, Avenue Jean-Baptiste Cl\'ement 93430
\\ Villetaneuse, France. \\ e-mail: {\tt
weissler@math.univ-paris13.fr}}
\end{center}

\begin{abstract}
 In this paper we prove local existence of solutions to the nonlinear
heat equation $u_t = \Delta u +a |u|^\alpha u, \; t\in(0,T),\;
x=(x_1,\,\cdots,\, x_N)\in \R^N,\; a = \pm 1,\; \alpha>0;$ with initial
value $u(0)\in L^1_{\rm{loc}}\left(\R^N\setminus\{0\}\right)$, anti-symmetric with respect to $x_1,\; x_2,\; \cdots,\; x_m$ and $|u(0)|\leq C(-1)^m\partial_{1}\partial_{2}\cdot \cdot \cdot
\partial_{m}(|x|^{-\gamma})$ for $x_1>0,\; \cdots,\; x_m>0,$ where $C>0$ is a constant, $m\in \{1,\; 2,\; \cdots,\; N\},$ $0<\gamma<N$ and $0<\alpha<2/(\gamma+m).$ This gives a local existence result
with highly singular initial values.

As an application, for $a=1,$ we establish new blowup criteria for $0<\alpha\leq 2/(\gamma+m)$, including the case $m=0.$ Moreover, if $(N-4)\alpha<2,$ we prove the existence of initial values
$u_0 = \lambda f,$ for which the resulting solution blows up in finite time $T_{\max}(\lambda f),$ if  $\lambda>0$ is sufficiently small. We also construct blowing up solutions with initial data $\lambda_n f$ such that $\lambda_n^{[({1\over \alpha}-{\gamma+m\over 2})^{-1}]}T_{\max}(\lambda_n f)$ has different finite limits along different
sequences $\lambda_n\to 0$. Our result extends the known ``small lambda'' blow up results for new values of $\alpha$ and a new class of initial data.
\end{abstract}

\medskip
\newpage

\section{Introduction}
\setcounter{equation}{0}
In this paper, which is a continuation of our previous article \cite{TW-I}, we study the existence and uniqueness of local in time solutions to the semilinear
heat equation
\begin{equation}
\label{NLheat}
u_t = \Delta u +a |u|^\alpha u,
\end{equation}
where $u=u(t,x)\in \R,\, t>0,\; x\in \R^N,\;  a = \pm 1,\; \alpha > 0,$ with highly singular initial data.  In addition, we obtain several new results on finite time blowup of solutions of \eqref{NLheat} in the case a = 1.

In \cite{TW-I} we considered initial data of the form
\begin{equation}
\label{ex1} u_0 = K(-1)^m\partial_{1}\partial_{2}\cdot
\cdot \cdot\partial_{m}\delta,
\end{equation}
in the case $a = 1,$ where $\delta$ is the Dirac point mass at the origin.
In this paper, our motivating example is the tempered distribution
\begin{equation}
\label{ex2}
u_0 = K (-1)^m\partial_{1}\partial_{2}\cdot \cdot
\cdot
\partial_{m}\left(|\cdot|^{-\gamma}\right) \in \Sd,
\end{equation}
where $m$ in an integer, $1 \le m \le N,$  $0<\gamma<N,$ and $K\in \R$.
In fact, we consider a more general class of initial data which  are in some sense bounded by \eqref{ex2}.
By a local solution, we mean a function $u : (0,T] \to C_0(\R^N)$, a classical solution of \eqref{NLheat} for $t\in (0,T],$ such that
$u(t) \to u_0$ in $ \mathcal{S}^\prime(\R^N)$ as $t \to 0$. Here $ \mathcal{S}^\prime(\R^N)$ is the space of tempered distributions on $\R^N,$ and
$$
C_0(\R^N)=\left\{f: \R^N \to \R \;{\rm continuous}\;; \lim_{|x|\to \infty}f(x)=0\right\}.
$$

As is standard practice, we study equation \eqref{NLheat} via the associated  integral formulation
\begin{equation}
\label{NLHint} u(t) = {\rm e}^{t\Delta}u_0 +a \int_0^t {\rm
e}^{(t-\sigma)\Delta} \big[|u(\sigma)|^\alpha u(\sigma)\big]
d\sigma,
\end{equation}
where ${\rm e}^{t\Delta}$ is the heat semigroup on $\R^N$  defined by
\begin{equation} \label{hsgrn}%
e^{t\Delta}\varphi=G_t \star \varphi,\; t>0,\; \varphi \in \Sd,
\end{equation}
$G_t$ being the Gauss kernel given by
\begin{equation}
G_t(x)=(4\pi t)^{-\frac{N}{2}} e^{-\frac{|x|^2}{4t}} ,\
 \; t>0 , \   \; x\in \Rd.\label{gk}%
\end{equation}
Throughout this paper, $\Omega \subset \R^N$ denotes the sector
\begin{equation} \label{dmn}%
\Omega= \biggl\{x=(x_1,\; x_2,\; \cdots,\; x_N) \in
\Rd;\; x_1>0,\; x_2>0,\cdots,\; x_m>0\biggr\},
\end{equation}
where $0\leq m\leq N$ is an integer. The case $m=0$ corresponds to $\Omega=\R^N.$  We write $\Omega$ rather than $\Omega_m$ to simplify notation
and since the value of $m$ will always be clear in context.  The important observation is that functions in $C_0(\R^N)$ which are anti-symmetric in
$x_1, x_2, \dots, x_m$ restrict to functions in $C_0(\Omega)$, and functions
in $C_0(\Omega)$ extend uniquely to functions in $C_0(\R^N)$ which are anti-symmetric in  $x_1, x_2, \dots, x_m$. Our basic approach, as in \cite{TW-I}, is to study equation \eqref{NLHint}
on $\Omega$ and extend the results by anti-symmetry to $\R^N$ (in the case $m \ge 1$).  One advantage of this approach is the ability to
use positivity and comparison properties on $\Omega$ which would not
hold directly on $\R^N$.  Of course, in the case $m = 0$ the antisymmetry condition
is vacuous.

Our first result directly complements the main result in \cite{TW-I} for the absorption case, i.e. \eqref{NLheat} with $a=-1.$

\begin{To}
\label{th1} Let $a=-1$ in \eqref{NLHint} and let the positive integer $m$ and the real number
$\alpha$ be such that $$1\leq m \leq N,\; 0<\alpha<{2\over N+m}.$$
Let $K\in \R,\; K\not=0,$ and set \begin{equation}
\label{initialdelta} u_0 = K(-1)^m\partial_{1}\partial_{2}\cdot
\cdot \cdot\partial_{m}\delta.
\end{equation} There exists a continuous solution $u
: (0,\infty) \to C_0(\R^N)$ of the integral equation \eqref{NLHint}
 with
initial value $u_0$ and such that $u(t)\to u_0$ in
$\mathcal{S}'(\R^N)$ as $t\to 0.$ Furthermore, $u(t)$ is anti-symmetric in
$x_1, x_2, \dots, x_m$, and if $K>0$ then $u(t,x) > 0$ for all $t > 0$ and $x \in \Omega$.
\end{To}

For $m=0$ the previous result is obtained in \cite[Theorem 3 and Remark 3, p. 82]{BF}. For $m=1$ and $\alpha<{1\over N+1}$
it is obtained in \cite[Corollary 11, p. 6110]{G1}. The result of \cite{G1} is stated for $a=1$ but the proof  is also valid for $a=-1.$ See also \cite{MT,Ri,Wu}
and references therein for the existence of solution to the nonlinear heat equation with other singular initial values.

The main purpose  of this paper is to consider initial values of the form \eqref{ex2}, as well as a more general class of singular initial values.
We have obtained the following result for initial data given by \eqref{ex2}.

\begin{To}
\label{th3} Let the positive integer $m$ and the real numbers
$\alpha,\; \gamma$ be such that
$$
1\leq m \leq N,\; 0<\gamma<N, \; 0<\alpha<{2\over \gamma+m}.
$$
Let $K\in \R,\; K\not=0,$ and set
\begin{equation}
\label{initialmodule} u_0 = K(-1)^m\partial_{1}\partial_{2}\cdot
\cdot \cdot\partial_{m}\left(|\cdot|^{-\gamma}\right) \in \Sd.
\end{equation}
There exist $T>0$ and a continuous solution
$u: (0,T] \to C_0(\R^N)$ of the integral equation \eqref{NLHint}
 with
initial value $u_0$ and such that $u(t)\to u_0$ in
$\mathcal{S}'(\R^N)$ as $t\to 0.$
Moreover, $u(t)$ is anti-symmetric in
$x_1, x_2, \dots, x_m$, and if $K>0$ then $u(t,x) > 0$ for all $t \in (0,T]$ and $x \in \Omega$.

These solutions can be extended to maximal solutions
$u : \left(0,T_{max}(u_0)\right) \to C_0(\R^N)$. In the case $a = -1$, $T_{max}(u_0) = \infty$.
In the case $a = 1$, $T_{max}(u_0) < \infty$.
\end{To}

\begin{rem}
\label{rem32b1} {\rm
The condition $\alpha<2/(\gamma+m)$ is in some sense optimal. See Proposition \ref{optimalxrho} below for an explanation.}
\end{rem}

\begin{rem}
\label{rem32b3}
  {\rm  The upper bound $\alpha<2/(\gamma+m)$ in Theorem~\ref{th3} for which
a solution exists with initial value given by (\ref{ex2})
 is precisely as predicted by a formal scaling argument. The initial value
$\partial_{1}\partial_{2}\cdot \cdot \cdot
\partial_{m}\left(|\cdot|^{-\gamma}\right),\; 0<\gamma<N,$ is in the Sobolev space $H^{s,r}(\R^ N)$, when
$$
s+m + \gamma < {N\over r}<\gamma.
$$
On the other hand, the scaling
critical exponent $s_c$ is given by $s_c={N\over r}-{2\over
\alpha}.$ To see this, we observe that the transformation
$\lambda^{{2\over \alpha}}u(\lambda^2t,\lambda x)$, which leaves
invariant the set of solutions to $u_t=\Delta u+a|u|^\alpha u$, acts
on initial values as $\lambda^{{2\over \alpha}}u_0(\lambda\cdot)$.
Furthermore,
$$\|\lambda^{{2\over \alpha}}u_0(\lambda\cdot)\|_{\dot{H}^{s,r}(\R^N)}=\lambda^{{2\over \alpha}-{N\over
r}+s}\|u_0\|_{\dot{H}^{s,r}(\R^N)},\;\forall\; \lambda >0.$$ In
other words, the homogeneous Sobolev norm of $\dot{H}^{s,r}(\R^N)$
is invariant under the action of $\lambda^{{2\over
\alpha}}u_0(\lambda\cdot)$ precisely if $s = s_c$.  We expect that solutions exist
if the initial value is in some
 $H^{s,r}(\R^N)$ with $s>s_c.$ With the initial value given  by (\ref{ex2}), this is possible when  ${N\over r}-{2\over
\alpha}<{N\over r}-(\gamma+m)$ i.e. $\alpha<2/(\gamma+m).$ A similar explanation for initial data given by \eqref{ex1} is given in \cite{TW-I}.}
\end{rem}

We have also obtained a local existence result for more general class of initial values, in some sense ``bounded by'' the distribution \eqref{ex2}.

\begin{To}
\label{th3bb}
Let the positive integer $m$ and the real numbers
$\alpha,\; \gamma$ be such that
$$
 1\leq m \leq N,\; 0<\gamma<N, \; 0<\alpha<{2\over \gamma+m}.
 $$
Let $u_0\in L^1_{\rm{loc}}\left(\R^N\setminus\{0\}\right)$ be anti-symmetric with respect to $x_1,\; x_2,\; \dots,\; x_m$ and satisfy  \begin{equation}
\label{initialinfmodule} |u_0(x)| \leq C(-1)^m\partial_{1}\partial_{2}\cdot
\cdot \cdot\partial_{m}\left(|x|^{-\gamma}\right),\; x \in \Omega,
\end{equation}
where $C>0$ is a constant. There exist $T>0$ and a continuous solution $u
: (0,T] \to C_0(\R^N)$ of the integral equation \eqref{NLHint}
 with
initial value $u_0$ and such that $u(t)\to u_0$ in
$\mathcal{S}'(\R^N)$ as $t\to 0.$
Moreover, $u(t)$ is anti-symmetric in
$x_1, x_2, \dots, x_m$, and if $u_0(x) \ge 0$ ($u_0(x) \not\equiv 0$) for $x \in \Omega$, then $u(t,x) > 0$ for all $t \in (0,T]$ and $x \in \Omega$.

These solutions can be extended to maximal solutions
$u : \left(0,T_{max}(u_0)\right) \to C_0(\R^N)$. In the case $a = - 1$, $T_{max}(u_0) = \infty$.  In the case $a = 1$,
$\liminf_{\lambda \to 0}\lambda^{[({1\over \alpha}-{\gamma+m\over 2})^{-1}]}T_{max}(\lambda u_0) > 0$.
\end{To}

See Theorem \ref{th3'} below (in Section 3) for a more general and precise version of the last result, including uniqueness and continuous dependence.
Uniqueness of the solutions constructed in Theorems  \ref{th3} and \ref{th3bb} holds in the space of functions anti-symmetric in
$x_1, x_2, \dots, x_m$ and bounded by a multiple of the solution of the linear heat equation
with the initial value given by \eqref{initialmodule}. See Sections \ref{xmgamma} and \ref{Wmgamma} for more details.

\begin{rem}
\label{remfw1}
 {\rm The case $m = 0$ in Theorem \ref{th3bb}  is  known. See \cite[Theorem 2.8, p. 313]{D} for $a = 1$ and  \cite[Theorem 8.8, p. 536]{CDEW} for
 $a = -1.$}
\end{rem}

We now turn to the blowup results, and for this discussion we of course take $a = 1$ in \eqref{NLHint}.  The first result is related to the decay of the initial data at infinity.
Let $u\in C\left(\left[0,T_{\max}(u_0)\right)\right)$ be the solution of \eqref{NLheat} with initial data $u_0\in C_0(\R^N),\; u_0\geq 0$. It is shown in \cite[Theorem 3.2 (i), p. 372]{LeeNi} that
there exists $C>0$ such that if
\begin{equation}
\label{conditionLeeNi}\liminf_{|x|\to \infty}|x|^{2/\alpha}u_0(x)>C,
\end{equation}
then $T_{\max}(u_0)<\infty$. This last condition is improved in \cite[Theorem 2, p. 66]{SW} to be
\begin{equation}
\label{conditionWeisslerSouplet}\liminf_{|x|\to \infty,\; x\in D}|x|^{2/\alpha}u_0(x)>C,
\end{equation}
where $C=C(D)>0$ and $D$ is a conic sector of $\R^N.$ The result of \cite{SW} is valid for $u_0$ nonnegative belonging to the Sobolev space $W^{1,s}_0(O), \; s$ large and  $O\subset \R^N$  a regular domain containing a sectorial domain $D.$
In \cite{MY}, the condition \eqref{conditionWeisslerSouplet} is both improved and weakened. The improvement is that domains $D$ which become
thin compared to a sector as $|x| \to \infty$ are allowed, but the result only applies to $u_0 \in L^\infty(\R^N)$ such that
\begin{equation}
\label{conditionMY}\liminf_{|x|\to \infty,\; x\in D}|x|^\gamma u_0(x)>C,
\end{equation}
for some $0 < \gamma < 2/\alpha$, and $C > 0$ {\it any} positive constant.  In this case, if $u_0 \ge 0$, then
$T_{\max}(u_0) < \infty$.  A similar result holds if
$u_0 \in W^{1,\infty}(\R^N)$, $u_0$ not necessarily positive, but with an additional condition on $\nabla u_0$.   See \cite[Theorem 2.1, p. 1435]{MY}, and the remark
which follows.  Finally, the result of \cite{MY} is improved in \cite[Theorems 1 and 2, p. 1019]{R}
in the case of positive initial value by giving more precise conditions on the initial
value, including the ``critical case", and by extending the type of domains considered.

We are able to extend and improve some of these blowup criteria.  As an example,
it is a consequence of the next theorem that if $u_0\in C_0(\R^N),\; u_0\geq 0$,
is such that $u_0(x) \ge c|x|^{-2/\alpha}\sin^2(\log |x|)$ for $c > 0$ large enough
and for all large $|x|$, then $T_{\max}(u_0)<\infty$.  See Remark~\ref{compact}(iii) just below for details and Remark~\ref{exampleinitialdata}  for more exotic examples.  As far as we are aware, this example is not covered by
any previously known results.  On the other hand, it seems that our methods do
not recover the results in \cite{MY,R} on domains which are asymptotically
thinner than sectors.

Throughout this paper,
we let $D_\lambda$ denote the dilation operator
\begin{equation}
\label{dilation}
(D_\lambda f)(x) = f(\lambda x)
\end{equation}
for all $\lambda > 0$.
For future reference, we recall the commutation relation between
the operator $D_\lambda$ and the heat semigroup on $\R^N$,
\begin{equation}
\label{commutation}
D_\lambda e^{\tau\lambda^2\Delta} = e^{\tau\Delta}D_\lambda
\end{equation}
for all $\lambda > 0$ and $\tau > 0$.

\begin{To}
\label{blowup} Let $a = 1$ in equation \eqref{NLHint} and let the integer $m$ and the real numbers
$\alpha,\; \gamma$ be such that $$ 0\leq m \leq N,\; 0<\gamma<N, \;
0 <\alpha \le {2\over \gamma+m}.$$
Let $u_0\in C_0(\R^N)$, anti-symmetric with respect to $x_1,\; x_2,\; \dots,\; x_m$, and let $u\in C\left(\left[0,T_{\max}(u_0)\right), C_0\left(\R^N\right)\right)$ be
the maximal solution of \eqref{NLheat} with initial data $u_0$.  Let  $f\in C_0(\R^N)$ be such that
\begin{itemize}
\item[(i)]$f$ is anti-symmetric with respect to $x_1,\; x_2,\; \dots,\; x_m$,
\item[(ii)]$u_0(x) \ge f(x)\geq 0$ for $x\in \Omega$,
\item[(iii)]there exists $K>0$ such that
\begin{equation}
\label{initialinfmodulef} f(x) \le K x_1x_2\cdots x_m|x|^{-\gamma-2m},\;  x\in \Omega.
\end{equation}
\end{itemize}
Suppose there exist a sequence  $\lambda_n \to \infty$ and
$z \in L^1_{\rm{loc}}\left(\R^N\setminus\{0\}\right)$, $z \not\equiv 0$,
 anti-symmetric with respect to $x_1,\; x_2,\; \dots,\; x_m$, such that
\begin{equation}
\label{subseqlim}
\lambda_n^{\gamma + m}D_{\lambda_n}f_{|\Omega} \to
z_{|\Omega}
\end{equation}
in ${\mathcal D'}\left(\Omega\right)$, as $n \to \infty$.
If $\alpha = {2\over \gamma+m}$, assume in addition that
\begin{equation}
\label{notsmall}
\|e^{\Delta}z\|_{L^\infty(\R^N)} > \left(\frac{1}{\alpha}\right)^{\frac{1}{\alpha}}.
\end{equation}
Then $T_{\max}(u_0)<\infty$.
\end{To}

\begin{rems}
 \label{compact}
 $\;${\rm
\begin{itemize}
 \item [(i)] By comparison on $\Omega$, it suffices to prove this result with $u_0 = f$.
Furthermore, Theorem~\ref{blowup} applies not just to $u_0 \in C_0(\R^N)$,
 but to any anti-symmetric $u_0\in L^1_{\rm{loc}}\left(\R^N\setminus\{0\}\right)$, with $u_0 \ge f$ on $\Omega$, for which a local solution of equation
 \eqref{NLHint}, with $a = 1$, exists (and is nonnegative on $\Omega$).  This includes, for example, the initial data prescribed in Theorem~\ref{th3bb}.
\item [(ii)]It follows from the considerations in \cite[Section 2]{CDW-DCDS} and \cite[Section 3]{MTW} (see also \cite[Section 5]{CDWsurvey}) that under condition \eqref{initialinfmodulef},
then for any sequence $\lambda_n \to \infty$ there is a subsequence, which we still call
 $\lambda_n$, and $z \in L^1_{\rm{loc}}\left(\R^N\setminus\{0\}\right)$,
 anti-symmetric with respect to $x_1,\; x_2,\; \dots,\; x_m$,
 such that \eqref{subseqlim} holds.  The issue in Theorem~\ref{blowup} is that we need $z\not\equiv 0$, or large in some sense if  $\alpha = {2\over \gamma+m}$.  Also, in the case $m \ge 1$ we need to show that
$e^{\Delta}z$ is well defined.  This follows from
the results in \cite{MTW} and will be explained in Section~\ref{blowup-span}.
\item [(iii)]In the case $m = 0$, i.e. $\Omega = \R^N$ (and so the anti-symmetry condition is vacuous), taking $\alpha = 2/\gamma$, one can see that if
$f\in C_0(\R^N)$ is such that $f(x) = c|x|^{-2/\alpha}$ for large $x$,
for example $|x| \ge \rho$,
then $\lambda^{2/\alpha}D_\lambda f(x) = c\lambda^{2/\alpha}|\lambda x|^{-2/\alpha} = c|x|^{-2/\alpha}$ for all $|x| \ge \rho/\lambda$.
Thus $\lambda^{2/\alpha}D_\lambda f \to c|\cdot|^{-2/\alpha}$ in
${\mathcal D'}(\R^N)$,  as $\lambda\to \infty$  and we recover the
result of \cite{LeeNi}. Furthermore, if $f(x) = c|x|^{-2/\alpha}\sin^2(\log |x|)$ for all $|x| \ge \rho$, then
$\lambda^{2/\alpha}D_\lambda f(x) = c|x|^{-2/\alpha}\sin^2(\log |x| + \log\lambda) $ for all $|x| \ge \rho/\lambda$, and so one may take
$\lambda_n = e^{\pi n}$ in order to apply to above result.
\item [(iv)] If $\alpha\leq 2/(N+m)$ then the anti-symmetry of $u_0$ and the condition
 $u_0 \ge 0$, $u_0 \not\equiv 0$ on $\Omega$ are enough
 to guarantee $T_{\max}(u_0)<\infty$. This is true thanks to Fujita-type  results
 (\cite{BL,M1,M2,K}) on the sector $\Omega$.  In other words, the above result is only of interest in the range ${2\over N+m} < \alpha \le {2\over \gamma+m}$, and it is the case
 $\alpha = {2\over \gamma+m}$ which is of most interest.
\end{itemize}
 }
 \end{rems}

The other main blowup results in this article concern initial values
$u_0 = \lambda f$, where $f \in C_0(\R^N)$ is anti-symmetric in $x_1, x_2, \dots, x_m$, but $f$  is not necessarily positive on $\Omega$.
We give conditions on $f$ which imply that $T_{\max}(u_0) = T_{\max}(\lambda f) < \infty$ for all sufficiently small $\lambda > 0$.  Results of this
type were first obtained by Dickstein \cite{D}, and our methods are inspired by those in \cite{D}.  In addition, we obtain information about the
asymptotic behavior of  $T_{\max}(\lambda f)$ as $\lambda\searrow 0$
and as $\lambda\nearrow \infty$.

Before stating the principal new result of this type,  we complement our previous result of \cite[Theorem 1.2, p. 508]{TW-I}.

\begin{To}
\label{lifespanderiveedelta}
Let $a = 1$ in \eqref{NLHint} and let the positive integer $m$ and the real number
$\alpha$ be such that $$1\leq m \leq N,\; 0<\alpha<{2\over N+m}.$$
Let $f \in C_0(\R^N)$ be anti-symmetric with respect to
$x_1,x_2,\cdots, x_m$, and suppose that there exist $t_0>0$ and
$K>0$ such that
\begin{equation}
\label{tzero}
|f(x)|\leq K |\partial_{1}\partial_{2}\cdot \cdot \cdot
\partial_{m} G_{t_0}(x)|,\forall \; x\in \R^N.
\end{equation}
Moreover, assume that
\begin{equation}\label{gho}
K_0=\int_{\R^N}x_1x_2\cdots x_mf(x)dx \neq 0.
\end{equation}
Let $u_\lambda: [0,T_{\max}(\lambda f))\to C_0(\R^N)$ be the maximal solution of the integral equation \eqref{NLHint} with initial value $\lambda f$. By \cite[Theorem 1.2, p. 508]{TW-I}, $T_{\max}(\lambda f) < \infty$  for
$\lambda > 0$ sufficiently small. In addition, we have
$$\lim_{\lambda\searrow  0}\lambda^{[({1\over \alpha}-{N+m\over 2})^{-1}]}T_{\max}(\lambda f)=T_{\max}(u_0),$$
where
$u_0=K_0(-1)^m\partial_{1}\partial_{2}\cdot
\cdot \cdot\partial_{m}\delta,$
and $T_{\max}(u_0) < \infty$ is the maximal existence time of the solution $u$ of \eqref{NLHint} with initial data $u_0,$ constructed by \cite[Theorem 1.1, p. 506]{TW-I}.
\end{To}
The previous result for $m=0$ is obtained in \cite[Theorem 1.4, p. 307]{D}.

We now consider initial data having polynomial decay at infinity. In \cite{LeeNi} the asymptotic behavior of the life span $T_{\max}(\lambda f)$ is studied  for small $\lambda>0$
and with $f \in C_b(\R^N)$,  $f\geq 0$ under the condition $\alpha \le 2/N$. (In this case $T_{\max}(\lambda f) < \infty$ by the Fujita results.)
It is shown in \cite[Theorem 3.15 (ii), p. 375]{LeeNi} and \cite[Theorem 3.21 (ii), p. 376]{LeeNi}, among other results,  that if $\alpha < 2/N$, $0 < \gamma < N$, and
\begin{equation}
\label{boundfmzero}
0 < \liminf_{|x|\to \infty}|x|^{\gamma}f(x) \le \limsup_{|x|\to \infty}|x|^{\gamma}f(x) < \infty,
\end{equation}
then
\begin{equation*}
0 < \liminf_{\lambda\to 0}\lambda^{[({1\over \alpha}-{\gamma\over 2})^{-1}]}T_{\max}(\lambda f) \le \limsup_{\lambda\to 0}\lambda^{[({1\over \alpha}-{\gamma\over 2})^{-1}]}T_{\max}(\lambda f) < \infty.
\end{equation*}
In \cite[Theorems 2.3 and 2.4, p. 1436]{MY} it is shown, for a certain class of $f$ (not necessarily positive) which satisfy the condition \eqref{conditionMY} with $0 < \gamma < 2/\alpha$ on the same type of ``thin" domain
used in \cite[Theorem 2.1, p. 1435]{MY} cited earlier, that $T_{\max}(\lambda f ) < \infty$ for all $\lambda > 0$ and
$$
\lambda^{\varepsilon}\leq\lambda^{[({1\over \alpha}-{\gamma\over 2})^{-1}]}T_{\max}(\lambda f)\leq \lambda^{-\varepsilon}\; \mbox{as}\; \lambda \searrow 0,
$$
for all $\varepsilon>0.$ In \cite[Theorem 1.3, p. 307]{D} it shown that if
$0 < \gamma <  N$, $\alpha < 2/\gamma$, $(N-2)\alpha < 4$ and
$\lim_{|x|\to \infty}|x|^{\gamma}f(x)=c \neq 0$ (without imposing that $f\geq 0$) then
$$\lim_{\lambda\searrow 0}\lambda^{[({1\over \alpha}-{\gamma\over 2})^{-1}]}T_{\max}(\lambda f)=C>0,$$
where $C>0$ is a finite constant.

The following two theorems extend and improve some of the above mentioned results.

\begin{To}[new Dickstein-type blow-up results]
\label{Dickstein-type} Let $a=1$ in \eqref{NLHint} and let the  integer $m$ and the real numbers $\alpha,\; \gamma$ be such that
$$0\leq m\leq N,\;  0<\gamma <N,\; 0<\alpha<{2\over \gamma+m},\; (N-2)\alpha<4.$$
Let $f\in C_0(\R^N)$, not necessarily positive, be anti-symmetric with respect to $x_1,\;x_2,\;\cdots,\; x_m$, and suppose that
\begin{equation}
\label{boundf}
0 < \liminf_{|x|\to \infty, \, x\in \Omega}\frac{|x|^{\gamma + 2m}}{ x_1\cdots x_m} f(x)\leq \limsup_{|x|\to \infty, \, x\in \Omega}\frac{|x|^{\gamma + 2m}}{ x_1\cdots x_m}f(x) < \infty.
\end{equation}
Then there exists $\lambda_0>0$ such that for all $\;0<\lambda<\lambda_0,$ the maximal solution $u_\lambda: [0,T_{\max}(\lambda f))\to C_0(\R^N)$ of the integral equation \eqref{NLHint}
with initial value $\lambda f$ blows up in finite time. Moreover,
\begin{equation}
\label{limsupliminf}
0 < \liminf_{\lambda\searrow 0}\lambda^{[({1\over \alpha}-{\gamma+m\over 2})^{-1}]}T_{\max}(\lambda f)\leq \limsup_{\lambda\searrow 0}\lambda^{[({1\over \alpha}-{\gamma+m\over 2})^{-1}]}T_{\max}(\lambda f) < \infty.
\end{equation}
\end{To}

\begin{rems}$\;${\rm

\begin{itemize}
\item[(i)] Dickstein proved the first small-$\lambda$  blow-up results  in \cite{D}.  Theorem 1.1 in \cite{D} assumes $\alpha<2/N$ with $f\in L^1(\R^N)\cap C_0(\R^N)$ and $\int f\not=0.$
Subsequent results of this type were obtained in \cite{G1} for $\alpha<1/(N+1)$ for $f\in L^1(\R^N)\cap C_0(\R^N)$ and $\int x_1f\not=0$ and later in \cite{TW-I}
for $\alpha<2/(N+m)$ and $f\in L^1(\R^N)\cap C_0(\R^N)$ with $\int x_1\cdots x_m f\not=0.$
 \item[(ii)] Theorem \ref{Dickstein-type} broadens the range of initial data in the
 case $m = 0$ as treated in \cite[Theorems 1.2 and 1.3, pp. 306-307]{D}
 (but not for the single point blowup result in \cite[Theorem 1.3, p. 307]{D}).
 Furthermore,  it extends the results
 to the cases $1 \le m \le N$.     Theorem \ref{Dickstein-type} also improves the range of
 allowable $\alpha$ compared to \cite{TW-I}
since we may take $2/(N+m)\leq \alpha<2/(\gamma+m).$
Note that both \cite[Theorem 1.2, p. 306]{D} and Theorem \ref{Dickstein-type}
above allow values of $\alpha > 2/N$.
 \item[(iii)]  The hypothesis $(N-2)\alpha<4$  implies type I blowup, a condition
 used in the proof in an essential way.  This hypothesis also is needed in
\cite[Theorems 1.2 and 1.3, pp. 306-307]{D}.
\item[(iv)] If $\alpha>{2\over N}$ and  $\varphi\in L^{q_c}(\R^N),\; q_c={N\alpha\over 2}$, then by \cite[Theorem 3(b)]{W2}, $T_{\max}(\lambda\varphi)=\infty$ for $\lambda$ sufficiently small.
(The positivity condition in \cite{W2} is clearly not needed.) If we consider Theorem \ref{Dickstein-type}
above in the case $\alpha>{2\over N}$, it is clear that
$f \not\in L^{q_c}(\R^N)$.  Indeed, for $f$ to be in $L^{q_c}(\R^N)$ would require that $N(\gamma + m)\alpha/2 > N$, i.e. $\alpha > 2/(\gamma + m)$.
\item[(v)] Let $\alpha>{2\over \gamma+m}$.  There exist  $f\in C_0(\R^N)$,  anti-symmetric with respect to $x_1,\;x_2,\;\cdots,\; x_m$,
satisfying  \eqref{boundf} such that $T_{\max}(\lambda f)=\infty$ for $\lambda$ sufficiently small. This follows by Proposition \ref{optimalxrho} below and a reflection argument.  If in addition, $\alpha > 2/N$, it also follows since such $f$ are in $L^{q_c}(\R^N)$ as described in the part (iv) above.
\end{itemize}}
\end{rems}
The following result gives a refinement of Theorem~\ref{Dickstein-type}.

\begin{To}
\label{Dickstein-type2} Let $a = 1$ in \eqref{NLHint} and let the  integer $m$ and the real numbers $\alpha,\; \gamma$ be such that
$$0\leq m\leq N,\;  0<\gamma <N,\; 0<\alpha<{2\over \gamma+m},\; (N-2)\alpha<4.$$
Let $f\in C_0(\R^N)$, not necessarily positive on $\Omega$, be anti-symmetric with respect to $x_1,\;x_2,\;\cdots,\; x_m$, and suppose that
\begin{equation}
\label{initialinfmodulef2}
|f(x)| \le K x_1x_2\cdots x_m|x|^{-\gamma-2m},\;  x\in \Omega,
\end{equation}
for some $K>0$.
Suppose there exist a sequence  $\mu_n \to \infty$ and
$z \in L^1_{\rm{loc}}\left(\R^N\setminus\{0\}\right)$,
anti-symmetric with respect to $x_1,\; x_2,\; \cdots,\; x_m$,
 such that
\begin{equation}
\label{subseqlim2}
\mu_n^{\gamma + m}D_{\mu_n}f_{|\Omega}\to z_{|\Omega}
\end{equation}
in ${\mathcal D'}(\Omega)$, as $n \to \infty$, so that in particular,
$z$ satisfies the same estimate \eqref{initialinfmodulef2}.
Let $T_{\max}(z)$ be the maximal existence time of the solution
to \eqref{NLHint} with $u_0 = z$, as constructed in Theorem~\ref{th3bb}.
It follows that
\begin{equation}
\label{lifespanlim}
\lambda_n^{[({1\over \alpha}-{\gamma + m\over 2})^{-1}]}T_{\max}(\lambda_n f) \to T_{\max}(z)
\end{equation}
where $\lambda_n = \mu_n^{-[{2\over \alpha}-(\gamma + m)]} \to 0$.
In particular, if $T_{\max}(z) < \infty$, then $T_{\max}(\lambda_n f) < \infty$ for all sufficiently large $n$.
\end{To}

The following result shows that the second inequality in \eqref{limsupliminf} may be strict.
\begin{cor}
\label{nonexistlimit} Let $a=1.$ Let the  integer $m$ and the real numbers $\alpha,\; \gamma$ be such that
$$0\leq m\leq N,\;  0<\gamma <N,\; 0<\alpha<{2\over \gamma+m},\; (N-2)\alpha<4.$$
Then there exists $f\in C_0(\R^N)$ satisfying the hypotheses of Theorem \ref{Dickstein-type} such that
$$
\liminf_{\lambda\to 0}\lambda^{[({1\over \alpha}-{\gamma+m\over 2})^{-1}]}T_{\max}(\lambda f) < \limsup_{\lambda\to 0}\lambda^{[({1\over \alpha}-{\gamma+m\over 2})^{-1}]}T_{\max}(\lambda f).
$$
\end{cor}

\begin{rems}$\;${\rm
\begin{itemize}
\item[(i)]
As noted in Remark~\ref{compact}(ii), it follows from the considerations in \cite[Section 2]{CDW-DCDS} and \cite[Sections 3 and 5]{MTW} that under condition \eqref{initialinfmodulef2},
for any sequence $\mu_n \to \infty$ there is a subsequence, which we still call
 $\mu_n$, and $z \in L^1_{\rm{loc}}\left(\R^N\setminus\{0\}\right)$,
anti-symmetric with respect to $x_1,\;x_2,\;\cdots,\; x_m$,
 such that \eqref{subseqlim2} holds.  In fact, one can easily choose $f$ which admit two such limits $z_1$ and $z_2$ with
 $0 < T_{\max}(z_1) < T_{\max}(z_2) < \infty$.
 See \cite[Theorem 1.2 and Proposition 2.9]{CDW-DCDS} and \cite[Theorem 1.4]{MTW}.
For such a function $f$,
 $\lim_{\lambda\to 0}\lambda^{[({1\over \alpha}-{\gamma+m\over 2})^{-1}]}T_{\max}(\lambda f)$ does not exist and takes different finite nonzero values for different sequences $\lambda_n\to \infty.$
See the proof of Corollary~\ref{nonexistlimit} in Section~\ref{blowup-span}
for more details.
   \item[(ii)] The situation described in Corollary~\ref{nonexistlimit} is
   quite different from previously known life-span results.  For example, if
   $f\in C_b(\R^N)$, $f \ge 0$, and
   $\lim_{|x|\to \infty}f(x)=f_\infty > 0$, then
    $$
    \lim_{\lambda \to 0}\lambda^{\alpha}T_{\max}(\lambda f) = \frac{1}{\alpha}f_\infty^{-\alpha}.
    $$
    Furthermore,
   when $f\in C_b(\R^N)$, $f \ge 0$, $f \not\equiv 0$, then
     $$
     \lim_{\lambda \to \infty}\lambda^{\alpha}T_{\max}(\lambda f)
     = \frac{1}{\alpha}||f||_\infty^{-\alpha}.
     $$
 For these two results, see \cite[Theorem 1, p. 167]{GW}.  See Proposition~\ref{nonnonexistlimit} below for a result analogous to
  Corollary~\ref{nonexistlimit}, but with $\lambda \to \infty$, and for which
  $T_{\max}(\lambda f) \sim     \lambda^{[-({1\over \alpha}-{\gamma + m\over 2})^{-1}]}$
  as $\lambda \to \infty$, instead of $\lambda^{-\alpha}$.
  \item[(iii)]  If $\frac{2}{N} < \alpha < \frac{2}{\gamma + m}$, there exists
  $f$ satisfying the hypotheses of Theorem~\ref{Dickstein-type2}, $f_{|\Omega} \ge 0$,
  for which two functions $z_1$ and $z_2 \not\equiv 0$ can be obtained
  as limits as in \eqref{subseqlim2} and such that
  $0 < T_{\max}(z_1) < T_{\max}(z_2) = \infty$.
   In this case, there is a sequence $\lambda_n \to 0$ for which
   $$
   \lambda_n^{[({1\over \alpha}-{\gamma + m\over 2})^{-1}]}T_{\max}(\lambda_n f) \to T_{\max}(z_1) < \infty,
   $$
   and another sequence $\lambda'_n \to 0$ for which
   $$
   (\lambda'_n)^{[({1\over \alpha}-{\gamma + m\over 2})^{-1}]}T_{\max}(\lambda'_n f) \to \infty,
   $$
Moreover, since $f \ge 0$ on $\Omega$, it follows from
   Theorem~\ref{blowup} that $T_{\max}(\lambda f) < \infty$ for all $\lambda > 0$, and it must be a monotone function of $\lambda >0$ (by comparison).
 This example shows shows that $\lambda^{[({1\over \alpha}-{\gamma + m\over 2})^{-1}]}T_{\max}(\lambda f)$ need not be monotone,
   and can exhibit qualitatively different behavior along different sequences of $\lambda \to 0$.

   To justify the existence of such $z_1$ and $z_2$, see
    \cite[Lemma 2.7]{CDW-DCDS}, as well as the proof of \cite[Theorem 1.2, p. 1114]{CDW-DCDS}   and the proof of \cite[Theorem 1.4, p. 358]{MTW}.
In order to guarantee that $T_{\max}(z_2) = \infty$, it suffices to
require that $||z_2||_{L^{q_c}}$ be small, where $q_c = N\alpha/2$.
\item[(iv)] It is shown in \cite[Theorem 1.3, p. 307]{D} that if $\lim_{|x|\to \infty}|x|^\gamma f(x)=c>0$ then
\begin{equation*}
\lim_{\lambda\to 0}\lambda^{[({1\over \alpha}-{\gamma\over 2})^{-1}]}T_{\max}(\lambda f)
\end{equation*}
 exists and is finite. This is a special case of Theorem~\ref{Dickstein-type2} for $m = 0$.  For $1 \le m \le N$  Theorem~\ref{Dickstein-type2} implies that if $\lim_{|x|\to \infty}{|x|^{\gamma+2m}\over x_1x_2\cdots x_m}f(x)=c>0$ then
 \begin{equation*}
 \lim_{\lambda\to 0}\lambda^{[({1\over \alpha}-{\gamma+m\over 2})^{-1}]}T_{\max}(\lambda f)
 \end{equation*}
 exists and is finite.
  \item[(v)] It is possible that, using the ideas described in \cite[Section 7]{CDWsurvey}, one could construct initial values $f$ for which
  $\lambda^{\sigma}T_{\max}(\lambda f)$ admits different finite, nonzero limits along different subsequences $\lambda_n \to 0$, and for different values of $\sigma > 0$.
\end{itemize}
}
\end{rems}

 See also Remark~\ref{limitexist} below.

We close the introduction with a few words about the proofs.
As mentioned earlier, the basic approach is to study equation \eqref{NLHint}
on the domain $\Omega$ defined by \eqref{dmn}, and then to reformulate
the results for solutions on $\R^N$ which are anti-symmetric with respect to
$x_1, x_2, \dots, x_m$.

Theorem \ref{th1} and Theorem \ref{lifespanderiveedelta} follow from straightforward adaptations of the methods in \cite{TW-I}, using in addition in the case of Theorem \ref{lifespanderiveedelta},
 that  the blowing up solutions are type I \cite{GMS1,GMS2} and the continuity of the maximal time of existence shown in \cite{FMZ}.

 To prove the existence results Theorems \ref{th3} and \ref{th3bb},
 the methods of \cite{TW-I} do not directly apply.
Since the distribution \eqref{initialdelta} is in fact $0$ as an element of
${\mathcal D'}(\Omega)$, in \cite{TW-I}  the integral equation \eqref{NLHint} was studied on $\Omega$ in the form
 \begin{equation}
\label{NLHintTW-I} u(t) = \varphi(t)+ \int_0^t {\rm
e}^{(t-\sigma)\Delta} \big[|u(\sigma)|^\alpha u(\sigma)\big]
d\sigma,
\end{equation}
for $t > 0$, without any specified initial value.  In this paper, we need to make
explicit the initial values described in Theorem~\ref{th3bb},
(in terms of their restriction to $\Omega$).  The key hypothesis is that
\begin{equation}
\label{hypintro}
\int_0^T \left\|e^{t\Delta_\Omega}u_0\right\|_\infty^\alpha dt<\infty,
\end{equation}
 where  $u_0$ is given by \eqref{initialmodule}, but restricted to $\Omega$, and
  $e^{t\Delta_\Omega}$ is the heat semigroup in $C_0(\Omega)$.
 This condition enables one to do a contraction mapping argument to
 prove existence of solutions to the integral equation \eqref{NLHint} on
 $\Omega$ which are bounded by a multiple of $e^{t\Delta_\Omega}u_0$.
 We make extensive use of the results in \cite{MTW}, which is devoted
 to the study of $e^{t\Delta_\Omega}$ on the set of initial values of
 Theorem~\ref{th3bb} (again, restricted to $\Omega$).

 As for the blowup results,
the proof of Theorem \ref{blowup} is a relatively straightforward application
of \cite[Theorem 1]{W4} along with \cite[Theorem 1.1(ii)]{MTW}, or in the case $m = 0$,
\cite[Proposition 3.8(i)]{CDW-DCDS}.  It is interesting to note
that the current paper, along with \cite{LeeNi,SW,MY},
use four essentially different methods to obtain the same type of
results.  In \cite{LeeNi}, Kaplan's eigenvalue blowup method on a centered
ball is used;  in \cite{SW}, self-similar subsolutions are used; and
in \cite{MY} energy methods with self-similar variables are used.

The proofs of Theorem~\ref{Dickstein-type} and  Theorem~\ref{Dickstein-type2} follow the basic plan developed in
\cite{D}, using some results from \cite{GMS1,GMS2,FMZ} on type I blowup
and the continuity of the blowup time.  In addition, the local theory of
Theorem~\ref{th3bb}, including the continuous dependence result Theorem~\ref{th3'}, below, is essential.

The rest of this paper is organized as follows.  In Section \ref{xmgamma}, we prove well-posedness on sectors $\Omega$ for the nonlinear heat equation \eqref{NLHint}
with initial data bounded by the distribution \eqref{ex2}. See the definition of the space ${\mathcal{X}}$ given by \eqref{spc} below. We also establish some nonexistence,
blowup and global existence results.
See Theorem \ref{globalproperties} and Proposition \ref{optimalxrho} below.   In Section \ref{Wmgamma} we apply the results of Section \ref{xmgamma}
to solutions on $\R^N$ which are anti-symmetric in $x_1, x_2, x_3, \cdots, x_m$. In particular, we establish the local existence
  with initial data of the form \eqref{ex2} and prove Theorems \ref{th1}, \ref{th3} and \ref{th3bb}. In Section~\ref{blowup-span},
we prove Theorems \ref{blowup}, \ref{lifespanderiveedelta}, \ref{Dickstein-type}, \ref{Dickstein-type2} and Corollary \ref{nonexistlimit}.

\section{Well-posedness on sectors}
\label{xmgamma}
\setcounter{equation}{0}
The purpose of this section is to study well-posedness of the integral equation \eqref{NLHint}
on the sectors $\Omega \subset \R^N$ defined by \eqref{dmn},
with the integer $m \in [1,N]$.
To be more precise, we first
recall that the heat semigroup on $\Omega$, which we denote
$e^{t\Delta_\Omega}$,
 is given by
\begin{equation} \label{sgomg}%
e^{t \Delta_\Omega} f(x)= \displaystyle{\int_{\Omega}} K_t(x,y) f(y) dy ,\; t>0,
\end{equation}
where
\begin{equation} \label{hks}
K_t (x, y) = (4\pi t)^{-\frac{N}{2}} \displaystyle \prod_{j=m+1}^{N}
e^{-\frac{|x_j-y_j|^2}{4t}} \prod_{i=1}^{m}
\left[e^{-\frac{|x_i-y_i|^2}{4t}}-e^{-\frac{|x_i+y_i|^2}{4t}} \right].
\end{equation}
See, for example, \cite{TW-I}. It is well-known that $e^{t\Delta_\Omega}$
is a $C_0$ contraction, positivity preserving semigroup on $C_0(\Omega)$, where
$$C_0(\Omega)=\left\{f\in C\left(\overline{\Omega}\right),\;  f(x)=0,\; \forall\; x\in \partial \Omega,\;\mbox{ and } \lim_{|x|\to \infty}f(x)=0\right\}.$$
Thus, instead of \eqref{NLHint}, we write
\begin{equation}
\label{intequOmega}
 u(t) =  {\rm e}^{t\Delta_\Omega}v_0 + a\int_0^t {\rm e}^{(t-\sigma)\Delta_{\Omega}}
\big(|u(\sigma)|^\alpha u(\sigma)\big) d\sigma,
\end{equation}
where $a=\pm 1$ and $\alpha>0$.  The initial value $v_0$ is in $\Dd$, the space of distributions on $\Omega$, and belongs to a specific class of distributions defined below.  We study solutions $u \in C((0,T];C_0(\Omega))$ such that
$u(t)\to v_0$  in  $\Dd$ as $t\searrow 0$.

In order to define the class of initial values $v_0 = \psi$ which we consider in \eqref{intequOmega}, we recall the following notation and definitions from \cite{MTW}  and \cite{TW-I}.
For $0<\gamma<N$ and integer $0 \le m \le N$, we let $\psi_0\, :\, \Omega \rightarrow\, \R$ be given by
\begin{equation} \label{psi0}%
\psi_0(x) = c_{m, \gamma} x_1 \cdots x_m  \vert x \vert^{-\gamma-2m} ,\, \forall\; x\in \Omega,
\end{equation}
where
\begin{equation} \label{cstcmg}%
c_{m, \gamma}=\gamma (\gamma+2) \cdots (\gamma+2m-2).
\end{equation}
In the case $m = 0$, then $c_{m, \gamma} = 1$ and $\psi_0(x) = |x|^{-\gamma}$, for $x \neq 0$.  We are mainly interested in the case $m \ge 1$, but since some of our results are new
for $m = 0$, i.e. $\Omega = \R^N$, we need to include $m = 0$ in the basic definitions.
In what follows, some of the statements need to modified in a trivial way if $m = 0$.

Clearly  $\psi_0$ is homogenous of degree $-\gamma-m,\; \psi_0\in C^\infty(\Omega)$ but $\psi_0\not\in C_0(\Omega)$, being singular at the origin. Also,
$\psi_0(x)>0,\; \forall\; x\in \Omega.$ We define the Banach space
\begin{equation}
\label{spc}
{\mathcal X}= \left\{ \psi \in L^1_{\rm loc}(\Omega) ;\quad
\frac{\psi}{\psi_0} \in L^\infty(\Omega)\right\},
\end{equation}
 with the norm
\begin{equation}
\label{Xnorm}
  \|\psi\|_{{\mathcal X}}=\left\|{\psi\over \psi_0}  \right\|_{ L^\infty(\Omega)}
\end{equation}
for all $\psi \in {\mathcal X}$, so that
\begin{equation}
\label{Xnorm2}
|\psi| \le   \|\psi\|_{{\mathcal X}}\psi_0.
\end{equation}
The set ${\mathcal X}$ will be the space of initial data for the integral equation \eqref{intequOmega} on $\Omega$.
The norm used in \cite{MTW} for $\X$ differs from \eqref{Xnorm} by a factor. This has no effect on the results we will use from \cite{MTW}.
Note that with the choice of norm \eqref{Xnorm}, we have $\|\psi_0\|_\X=1.$

 By \cite[Theorem 1.1, p. 343]{MTW} the heat semigroup
$e^{t\Delta_{\Omega}}$ is well defined on $\X.$  In particular, there exists a constant $\mathcal{C}\ge 1$ such that
\begin{equation}
\label{3etoile}
\|e^{t\Delta_{\Omega}}\psi\|_\X\leq \mathcal{C}\|\psi\|_\X,\; \forall\; t>0,\; \forall\; \psi\in \X.
\end{equation}
Also $e^{t\Delta_{\Omega}}$  is
a bounded map from $\X$ to $C_0(\Omega)$ for all $t>0$. Consider the Banach space
\begin{equation}
\label{dualsp}
{\mathcal Y}:=\left\{ \psi \in L^1_{\mbox{loc}}(\Omega) ;\quad
\psi \psi_0\in L^1(\Omega)
\right\}.
\end{equation}
Clearly, ${\mathcal X}$ is the dual of ${\mathcal Y}$, that is ${\mathcal Y}'={\mathcal X}$. Let
\begin{equation} \label{ball}%
{\mathcal B}_{K}:=\biggr\{\psi \in \X;\; \; \|\psi\|_{{\mathcal X}} \leq  K \biggl\},
\end{equation}
where $K>0$.  Endowed with the
weak$^\star$ topology, ${\mathcal B}_{K}$, is compact (see \cite{brezis}) and metrizable since ${\mathcal Y}$ is separable.

\begin{defini}
\label{BKstar}
We denote by ${\mathcal B}_{K}^*$ the ball ${\mathcal B}_{K}$ endowed with
the weak$^\star$ topology on $\X.$
\end{defini}

We know that ${\mathcal B}_{K}^*$ is a compact metric
space, hence complete and separable, for all $K>0$. By \cite[Proposition 3.1, p. 356]{MTW}, we also know that if $(\psi_n)_{n = 1}^\infty$ is a sequence in
${\mathcal B}_{K}$ and $\psi\in {\mathcal B}_{K}$,
then $\psi_n\to \psi$ in ${\mathcal B}_{K}^*$  if and only if $\psi_n\to \psi$ in $\Dd.$

\begin{rem}{\rm The function $\psi_0$ is simply the distribution $u_0$ given by \eqref{ex2},
with $K = 1$, realized as a point function on $\Omega$.  The relationship
between $e^{t\Delta_\Omega}\psi_0$ and $e^{t\Delta}u_0$ was studied in
\cite{MTW} and will be crucial in the next section to interpret the results of
this section as results about \eqref{NLHint} on $\R^N$.}
\end{rem}

We now introduce the space in which we will construct the solutions of \eqref{intequOmega}.
We set
\begin{equation}
\label{Psi}
\Psi(t,x)=e^{t\Delta_\Omega} \psi_0(x),\; t>0,\; x\in \Omega,
\end{equation}
and we sometimes write $\Psi(t)=e^{t\Delta_\Omega} \psi_0.$
By \cite[Theorem 1.1, p. 343]{MTW}, it follows that
$\Psi: (0,\infty) \to C_0(\Omega)$ is a continuous map.
Since $ \psi_0(x) > 0$ for $x \in \Omega$, we also have
\begin{equation*}
\Psi(t,x) > 0, \, \forall t > 0, \, \forall x \in \Omega.
\end{equation*}
Moreover, it follows from  \eqref{Xnorm2}, \eqref{3etoile} and \eqref{Xnorm} that
\begin{equation}
\label{Psibd}
\Psi(t) =e^{t\Delta_\Omega} \psi_0
\le \|e^{t\Delta_\Omega} \psi_0\|_{{\mathcal X}}\psi_0
\le \mathcal{C}\|\psi_0\|_{{\mathcal X}}\psi_0
= \mathcal{C}\psi_0, \, \forall t > 0.
\end{equation}
For a given $T> 0$, let $X$
be  defined by
$$X = X_T = \biggl\{u\in C\left((0,T];C_0(\Omega)\right);\; \exists M > 0,\;  |u(t,x)| \leq M\Psi(t,x),\; \forall\; t\in (0,T],\; \forall\; x\in \Omega \biggr\}.$$
The space $X = X_T$ equipped with the norm
$$
|||u||| = |||u|||_{X_T} =  \sup_{0 < t \le T,\; x \in \Omega}{|u(t,x)|\over
\Psi(t,x)},
$$
is a Banach space.  Clearly,
\begin{equation}
\label{norm}|u(t)| \le |||u|||e^{t\Delta_\Omega} \psi \Psi(t),\; \mbox{ for all } \; 0 < t \le T.
\end{equation}
If needed for clarity, we may denote $X$ by $X_T$ or $ X_{T,\Psi}$ and the
norm $|||u|||$ by $|||u|||_{X_T}$ or even  $|||u|||_{X_{T,\Psi}}$.
In addition, we observe that if $\psi \in \X$, then
\begin{eqnarray}
\label{semigpbdd}
\nonumber |||e^{t\Delta_\Omega} \psi|||_{X_T}
&=& \sup_{0 < t \le T,\; x \in \Omega}{|e^{t\Delta_\Omega} \psi|\over\Psi(t,x)}
= \sup_{0 < t \le T,\; x \in \Omega}{|e^{t\Delta_\Omega}({\psi\over\psi_0})\psi_0|\over\Psi(t,x)}\\
&\le&\sup_{0 < t \le T,\; x \in \Omega}{|e^{t\Delta_\Omega}\psi_0|\over\Psi(t,x)}
\Big\|{\psi\over\psi_0}\Big\|_{L^\infty(\Omega)} = |||\psi|||_\X.
\end{eqnarray}

It also follows from \cite[Theorem 1.1, p. 343]{MTW} that $\|\Psi(t)\|_{L^\infty(\Omega)} \le Ct^{-(\gamma + m)/2}$
for all $t > 0$. In fact, by \cite[Relation (1.15), p. 344]{MTW} and since
$\psi_0$ is homogeneous of degree $-\gamma-m$, we have that
\begin{equation}
\label{etoile}
D_{\sqrt t}(e^{t\Delta_\Omega} \psi_0) = t^{-(\gamma + m)/2}e^{\Delta_\Omega} \psi_0,\; \forall \; t>0,
\end{equation}
where the operator $D_{\sqrt t}$ is given by \eqref{dilation}, so that
\begin{equation}
\label{triangle}
\|\Psi(t)\|_ {L^\infty(\Omega)} =\|e^{t\Delta_\Omega} \psi_0\|_ {L^\infty(\Omega)} =
\|D_{\sqrt t}(e^{t\Delta_\Omega} \psi_0)\|_ {L^\infty(\Omega)} = \|e^{\Delta} \psi_0\|_ {L^\infty(\Omega)} t^{-(\gamma + m)/2},\; \forall \; t>0.
 \end{equation}
It follows, and this is fundamental in our existence proofs, that
for any $0<\alpha<2/(\gamma+m)$,
\begin{equation}
\label{2.3deTW-I}\int_0^A \left\|\Psi(t)\right\|_{L^\infty(\Omega)} ^\alpha dt<\infty,\; \forall\; A>0.
\end{equation}

We can now state and prove our existence result for \eqref{intequOmega} with more general singular initial data.

\begin{To}[well-posedness in $\X$]
\label{th4} Let the  integer $m$ and the real numbers
$\alpha,\; \gamma$ be such that $$ 0\leq m \leq N,\; 0<\gamma<N, \; 0<\alpha<{2\over \gamma+m}.$$  Let $K > 0$,  $M>0$ and $T>0$ be such
that
\begin{equation}
\label{e21n} K+2(\alpha+1)M^{\alpha+1}\int_0^T \|\Psi(\sigma)\|_{L^\infty(\Omega)}^\alpha d\sigma\leq
M
\end{equation}
and
\begin{equation}
\label{e22n} 2(\alpha+1)M^{\alpha}\int_0^T \|\Psi(\sigma)\|_{L^\infty(\Omega)}^\alpha
d\sigma<1,
\end{equation}
where $\Psi$ is given by \eqref{Psi}. Let $a = \pm 1$. For every  $\psi \in \X$
with $||\psi||_{\X} \le K$  there exists a unique solution $u\in X_T$ of
\begin{equation}
\label{intequ} u(t) =  {\rm e}^{t\Delta_\Omega}\psi +a \int_0^t {\rm e}^{(t-\sigma)\Delta_{\Omega}}
\big(|u(\sigma)|^\alpha u(\sigma)\big) d\sigma
\end{equation}
such that $$|||u|||_{X_T}\leq M.$$ The integral in (\ref{intequ}) is absolutely convergent
in $C_0(\Omega).$

Furthermore, we have the following additional properties.
\begin{itemize}
\item[(i)] $u(t)\to \psi$ as $t\to 0$ in $L^1_{\rm loc}(\Omega)$, hence in $\Dd$ and in ${\mathcal B}_{M'}^*$, where $M'= \mathcal{C}M$, and ${\mathcal C}$ is the constant appearing
in \eqref{3etoile}.
\item[(ii)] Let  $\psi_1$ and $\psi_2$
be in ${\mathcal B}_{K}.$  Let $u_1 \in X_T$, with $|||u_1|||_{X_T}\leq M$, and
$u_2 \in X_T$, with $|||u_2|||_{X_T}\leq M$, be the solutions of (\ref{intequ}) with respectively initial values $\psi_1$ and $\psi_2.$ It follows that
\begin{equation}
\label{CD}
|||u_1-u_2|||_{X_T}\leq C |||{\rm e}^{t\Delta_\Omega}(\psi_1-\psi_2)|||_{X_T}\leq C  \|\psi_1-\psi_2\|_{\X},
\end{equation} where $C>0$ is the constant given  by \eqref{mathcalC} below.
\item[(iii)] Let $(\psi_k)_{k = 1}^\infty \subset {\mathcal B}_{K}^*$ and $\psi\in {\mathcal B}_{K}^*.$ Let $u_k \in X_T$, with $|||u_k|||_{X_T}\leq M$, be the corresponding solution of (\ref{intequ}) with initial data $\psi_k$ and let
$u \in X_T$, with $|||u|||_{X_T}\leq M$ be the corresponding solution of (\ref{intequ}) with initial data $\psi$. If $\psi_k\to \psi$ as $k\to \infty$ in  $\Dd$ (that is also in ${\mathcal B}_{K}^*$)
 then $u_k(t)\to u(t)$ as $k\to \infty$ in $C_0(\Omega)$ for all $t\in (0,T],$ and uniformly in $[t_0,T],\; \forall \; t_0\in (0,T).$
 \item[(iv)] Let $\psi \in {\mathcal B}_{K}^*,$ let $u$  be the solution of \eqref{intequ} with initial value $\psi$
 and $v$ be the solution of \eqref{intequ} with initial value $|\psi|.$ Then
 $$|u(t)|\leq v(t),\; \forall\; t\in (0,T].$$
 \item[(v)] If $\psi\geq 0$ then $u(t)\geq 0$ on $\Omega$ for all $t\in (0,T].$
 \item[(vi)] Let  $\psi_1$ and $\psi_2$
be in ${\mathcal B}_{K}.$  Let $u_1$ and $u_2$ be the solutions of \eqref{intequ} with respectively initial values $\psi_1$ and $\psi_2.$ If $\psi_1\leq \psi_2$ then $u_1(t)\leq u_2(t),\; \forall\; t\in (0,T].$
\end{itemize}
\end{To}

\begin{proof}
The proof uses and adapts some arguments of \cite{TW-I} in conjunction with some results established in \cite{MTW,CDW-DCDS}.  For simplicity, we consider only the case $m 
\ge 1$.  For $m = 0$, the results cited below in \cite{MTW} need to be replaced by
corresponding results in \cite{CDW-DCDS}.
Let $K>0,$  $M>0$ and $T>0$ satisfying (\ref{e21n}) and
(\ref{e22n}). Let
\begin{equation*}
B(0,M):=\{u\in X_T\;  :\;|||u|||_{X_T}\leq M\},
\end{equation*}
which is a complete metric space, with the metric
\begin{equation*}
|||u - v|||_{X_T}.
\end{equation*}
  Fix $\psi\in \X$ such that $\|\psi\|_\X\leq K.$  Define, for
$u\in B(0,M), $
$${\mathcal F}_{\psi}u(t)={\rm e}^{t\Delta_\Omega}\psi+a\int_0^t{\rm e}^{(t-\sigma)\Delta_{\Omega}}\big(|u(\sigma)|^\alpha u(\sigma)\big)d\sigma.$$
We will show that ${\mathcal F}_{\psi}$ is a contraction on $B(0,M).$  (For simplicity, we denote $\|\cdot \|_{L^\infty(\Omega)}$ by $\|\cdot\|_\infty.$) Let  $\psi_1,\; \psi_2\in {\mathcal B}_K$ and let $u_1,\; u_2\in B(0,M)$. Then
$$ {\mathcal
F}_{\psi_1}(u_1)(t)-{\mathcal F}_{\psi_2}(u_2)(t)= {\rm e}^{t\Delta_\Omega}(\psi_1-\psi_2)+a\int_0^t{\rm
e}^{(t-\sigma)\Delta_\Omega}\Big(|u_1|^\alpha u_1(\sigma)-|u_2|^\alpha
u_2(\sigma)\Big)d\sigma.$$ Using the fact that ${\rm e}^{t\Delta_\Omega}$ is positivity preserving, and (\ref{norm}),
we get that
\begin{eqnarray}\nonumber |{\mathcal F}_{\psi_1}(u_1)(t)-{\mathcal F}_{\psi_2}(u_2)(t)|&\leq &|{\rm e}^{t\Delta_\Omega}(\psi_1-\psi_2)|+\\ && \nonumber
\hspace{-3cm}(\alpha+1)\int_0^t{\rm
e}^{(t-\sigma)\Delta_\Omega}\Big(|u_1(\sigma)-u_2(\sigma)|\big(|u_1(\sigma)|^\alpha
+|u_2(\sigma)|^\alpha\big) \Big)d\sigma\\ \nonumber
&& \hspace{-4cm} \leq  {\rm e}^{t\Delta_\Omega}|\psi_1-\psi_2|+(\alpha+1)\int_0^t\Big({\rm e}^{(t-\sigma)\Delta_\Omega}
\Psi(\sigma){|u_1(\sigma)-u_2(\sigma)|\over \Psi(\sigma)}\Big)\times\\ \nonumber
&& \hspace{3cm}\Big(\|u_1(\sigma)\|_\infty^\alpha
+\|u_2(\sigma)\|_\infty^\alpha \Big)d\sigma\\ \nonumber
 && \hspace{-4cm} \leq  {\rm e}^{t\Delta_\Omega}|\psi_1-\psi_2|+(\alpha+1)\int_0^t\Big({\rm
e}^{(t-\sigma)\Delta_\Omega}
\Psi(\sigma)\Big)|||u_1-u_2|||_{X_T}\times\\ \nonumber
&& \hspace{3cm}\Big(\|u_1(\sigma)\|_\infty^\alpha
+\|u_2(\sigma)\|_\infty^\alpha \Big)d\sigma\\ \label{inqcontraction}
&&\hspace{-4cm} \leq {\rm e}^{t\Delta_\Omega}|\psi_1-\psi_2|+\Psi(t)\left(
2M^\alpha(\alpha+1)\int_0^t\|\Psi(\sigma)\|_\infty^\alpha
d\sigma \right)|||u_1-u_2|||_{X_T}.
\end{eqnarray}
Let now $\psi_1=\psi\in {\mathcal B}_K,\; u_1=u\in B(0,M)$ and $\psi_2=0,\; u_2=0.$ Then, by the previous inequality and \eqref{Xnorm2}, we get

\begin{eqnarray*}
|{\mathcal F}_{\psi}u(t)|&\leq&
{\rm e}^{t\Delta_\Omega}|\psi|+\Psi(t)\left(
2M^\alpha(\alpha+1)\int_0^t\|\Psi(\sigma)\|_\infty^\alpha
d\sigma \right)|||u|||_{X_T}\\
&\leq&
{\rm e}^{t\Delta_\Omega}\|\psi\|_{\X}\psi_0+\Psi(t)\left(
2M^{\alpha+1}(\alpha+1)\int_0^t\|\Psi(\sigma)\|_\infty^\alpha
d\sigma \right)\\
&\leq& K\Psi(t)+\Psi(t)\left(
2M^{\alpha+1}(\alpha+1)\int_0^t\|\Psi(\sigma)\|_\infty^\alpha
d\sigma \right)\\
&\leq& \Psi(t)\left[K+
2M^{\alpha+1}(\alpha+1)\int_0^t\|\Psi(\sigma)\|_\infty^\alpha
d\sigma \right].
\end{eqnarray*}
That is,
$${|{\mathcal F}_{\psi}u(t)|\over \Psi(t)}\leq K+2M^{\alpha+1}(\alpha+1)\int_0^T\|\Psi(\sigma)\|_\infty^\alpha
d\sigma,\; \forall t\in (0,T].$$ Thus by (\ref{e21n}),
$|||{\mathcal F}_{\psi}u|||_{X_T}\leq M$, and hence ${\mathcal F}_{\psi}$ maps $B(0,M)$
into itself.

Let now $\psi_1=\psi_2=\psi\in {\mathcal B}_K$ in \eqref{inqcontraction}.
This gives
\begin{eqnarray*}
|||{\mathcal F}_{\psi}(u_1) -{\mathcal F}_{\psi}(u_2)|||_{X_T} =
\sup_{0 < t \le T}{|{\mathcal F}_{\psi}(u_1)(t)-{\mathcal F}_{\psi}(u_2)(t)|\over \Psi(t)}\\
\leq
\Big(2(\alpha+1)M^\alpha \int_0^T \|\Psi(\sigma)\|_\infty^\alpha
d\sigma\Big)|||u_1-u_2|||_{X_T}.
\end{eqnarray*}
 Then by (\ref{e22n}), ${\mathcal F}$ is a strict contraction from
$B(0,M)$ into itself. It follows, by the  Banach fixed point theorem that there
exists a unique solution $u \in X_T$ of the integral equation (\ref{intequ})
such that $|||u|||_{X_T}\leq M.$

Let $u$ be the solution of the integral equation constructed by the above fixed point argument.
Then
\begin{eqnarray}
\label{absconv}
\nonumber\int_0^t \|{\rm e}^{(t-\sigma)\Delta_\Omega} |u(\sigma)|^\alpha u(\sigma)\|_\infty d\sigma &\leq&
\int_0^t \left\|({\rm e}^{(t-\sigma)\Delta_\Omega} |u(\sigma)|) \|u(\sigma)\|^\alpha_\infty\right\|_\infty d\sigma\\
 &\leq& M^{\alpha+1}\|\Psi(t)\|_\infty\int_0^t
\|\Psi(\sigma)\|_\infty^\alpha d\sigma,
\end{eqnarray}
which shows that the integral in (\ref{intequ}) is absolutely convergent
in $C_0(\Omega)$.

We now prove the additional properties.

(i) By \eqref{norm} and \eqref{Psibd},
$$
|u(t)| \le |||u|||_{X_T} \Psi(t)\leq M\mathcal{C}\psi_0,\; \mbox{ for all } \; 0 < t \le T.
$$
Hence $u(t)\in \X$ for all $t\in (0,T]$ and $\|u(t)\|_\X\leq M\mathcal{C}$,
where $\mathcal{C}$ is given by \eqref{3etoile}.
On the one hand we have that
$$u(t)-{\rm e}^{t\Delta_\Omega}\psi=a \int_0^t {\rm e}^{(t-\sigma)\Delta_{\Omega}}
\big(|u(\sigma)|^\alpha u(\sigma)\big) d\sigma,$$
and, again using \eqref{norm} and \eqref{Psibd},
\begin{eqnarray*}
\left|\int_0^t {\rm e}^{(t-\sigma)\Delta_{\Omega}}
\big(|u(\sigma)|^\alpha u(\sigma)\big) d\sigma\right|&\leq & M^{\alpha+1}\int_0^t {\rm e}^{(t-\sigma)\Delta_{\Omega}}
\big(\Psi(\sigma)|\Psi(\sigma)|^\alpha \big) d\sigma \\ &\leq & M^{\alpha+1}\int_0^t \big({\rm e}^{(t-\sigma)\Delta_{\Omega}}\Psi(\sigma)\big)
\|\Psi(\sigma)\|^\alpha_\infty  d\sigma \\ &= & M^{\alpha+1} \Psi(t)\int_0^t \|\Psi(\sigma)\|^\alpha_\infty  d\sigma \\ &\leq & {\mathcal{C}}M^{\alpha+1} \psi_0 \int_0^t \|\Psi(\sigma)\|^\alpha_\infty
d\sigma\to 0 ,\; \mbox{as}\; t\to 0,
\end{eqnarray*}
in $L^1_{\rm loc}(\Omega)$, hence in $\Dd$.
On the other hand, by \cite[Proposition 2.5, p. 353]{MTW} and Remark \ref{error1} below, ${\rm e}^{t\Delta_\Omega}\psi\to
\psi$ as $t\to 0$ in $L^1_{\rm loc}(\Omega)$, hence in $\Dd$.
 We thus conclude that $u(t)\to \psi$ as $t\to 0$ in $L^1_{\rm loc}(\Omega)$, hence in $\Dd$. The convergence in ${\mathcal B}_{M'}^*$ follows by \cite[Proposition 3.1 (i), p. 356]{MTW}.

(ii)  By \eqref{inqcontraction}, \eqref{norm} and the fact that ${\mathcal
F}_{\psi_1}(u_1)=u_1$ and ${\mathcal
F}_{\psi_2}(u_2)=u_2$, we have that
\begin{eqnarray*}
|u_1(t)-u_2(t)|&\leq & |{\rm e}^{t\Delta_\Omega}(\psi_1-\psi_2)|+\Psi(t)\Big(2(\alpha+1)M^\alpha \int_0^T \left\|\Psi(\sigma)\right\|_\infty^\alpha
d\sigma\Big) |||u_1-u_2|||_{X_T}\\
&\leq &\Psi(t)|||{\rm e}^{t\Delta_\Omega}(\psi_1-\psi_2)|||_{X_T}+\Psi(t)\Big(2(\alpha+1)M^\alpha \int_0^T \left\|\Psi(\sigma)\right\|_\infty^\alpha
d\sigma\Big) |||u_1-u_2|||_{X_T}\\
&= & \Psi(t)\left[|||{\rm e}^{t\Delta_\Omega}(\psi_1-\psi_2)|||_{X_T}+\Big(2(\alpha+1)M^\alpha \int_0^T \left\|\Psi(\sigma)\right\|_\infty^\alpha
d\sigma\Big) |||u_1-u_2|||_{X_T}\right].\\
%&\leq & \|\psi_1-\psi_2\|_\X{\rm e}^{t\Delta_\Omega}\psi_0+\Psi(t)\Big(2(\alpha+1)M^\alpha \int_0^T \left\|\Psi(\sigma)\right\|_\infty^\alpha
%d\sigma\Big) |||u_1-u_2|||_{X_T}\\
%&= & \Psi(t)\left[\|\psi_1-\psi_2\|_\X+\Big(2(\alpha+1)M^\alpha \int_0^T \left\|\Psi(\sigma)\right\|_\infty^\alpha
%d\sigma\Big) |||u_1-u_2|||_{X_T}\right].
\end{eqnarray*}
Consequently,
$$|||u_1-u_2|||_{X_T}\leq |||{\rm e}^{t\Delta_\Omega}(\psi_1-\psi_2)|||_{X_T}+\Big(2(\alpha+1)M^\alpha \int_0^T \left\|\Psi(\sigma)\right\|_\infty^\alpha
d\sigma\Big) |||u_1-u_2|||_{X_T}.$$
Thus, by \eqref{e22n}, the first inequality in (\ref{CD}) follows  with
\begin{equation}
\label{mathcalC} C={1\over 1-\left(2(\alpha+1)M^\alpha \int_0^T\left\|\Psi(\sigma)\right\|_\infty^\alpha d\sigma\right)}.
\end{equation}
The second inequality in (\ref{CD}) then follows from \eqref{semigpbdd}.

(iii) Since $u_k \in X_T$, with $|||u_k|||_{X_T}\leq M$, it follows from
\cite[Lemma 2.6, p. 355]{MTW} that
\begin{equation}
\label{ukest}
|u_k(t,x)| \le M\Psi(t,x) \le C(t + |x|^2)^{-\frac{\gamma + m}{2}}, \quad t \in (0,T], \, x\in\Omega,
\end{equation}
for some $C > 0$ independent of $k$.  In particular, for any fixed $t \in (0,T]$,
the functions $u_k(t)$ are uniformly bounded in $C_0(\Omega)$.  It follows
by standard parabolic regularity arguments, along with \eqref{ukest}, that the sequence $(u_k)_{k = 1}^\infty$
is relatively compact in $C((t_0,T];C_0(\Omega))$, for any $0 < t_0 < T$.
By a diagonal argument, it follows that there exist $v \in C((0,T];C_0(\Omega))$,
and a subsequence $u_{k_n}$ such that $u_{k_n} \to v$
in $C((t_0,T];C_0(\Omega)$, for all $0 < t_0 < T$.  It follows that $v \in X_T$, with
$||v|||_{X_T}\leq M$.

We next observe that $v$ satisfies the integral equation \eqref{intequ}.
This follows easily since
\begin{equation*}
u_{k_n}(t) =  {\rm e}^{t\Delta_\Omega}\psi_{k_n} +a \int_0^t {\rm e}^{(t-\sigma)\Delta_{\Omega}}
\big(|u_{k_n}(\sigma)|^\alpha u_{k_n}(\sigma)\big) d\sigma,
\end{equation*}
for all $0 < t \le T$, and letting $n \to \infty$.
Indeed, ${\rm e}^{t\Delta_\Omega}\psi_{k_n} \to {\rm e}^{t\Delta_\Omega}\psi$
as $n \to \infty$, by \cite[Proposition 4.1(ii), p. 359]{MTW}.
Moreover, by the dominated convergence theorem, since
\begin{equation*}
||{\rm e}^{(t-\sigma)\Delta_\Omega} |u_{k_n}(\sigma)|^\alpha u_{k_n}(\sigma)\|_\infty \leq
 M^{\alpha+1}\|\Psi(t)\|_\infty\
\|\Psi(\sigma)\|_\infty^\alpha,
\end{equation*}
as in \eqref{absconv},
the integral terms converge to $a\int_0^t {\rm e}^{(t-\sigma)\Delta_{\Omega}}
\big(|v(\sigma)|^\alpha v(\sigma)\big) d\sigma$.

Since $v\in X_T$, $|||v|||_{X_T} \le M$, and $v$ satisfies the integral
equation \eqref{intequ}, it follows from uniqueness that $v = u$,
so $u_{k_n} \to u$
on $C((t_0,T];C_0(\Omega))$, for any $0 < t_0 < T$.
A standard subsequence argument shows that in fact $u_k \to u$
on $C((t_0,T];C_0(\Omega))$, for any $0 < t_0 < T$.

(iv) By \cite[Proposition 3.2, p. 357]{MTW}, we know that  ${\mathcal B}_{K}\cap {\mathcal{D}}(\Omega)$ is dense in ${\mathcal B}_{K}^*.$ Let
$(\psi_n)_{n = 1}^\infty\subset {\mathcal B}_{K}\cap \D$ such that $\psi_n\to \psi$ as $n\to \infty$ in ${\mathcal B}_{M'}^*,$ hence in $\Dd.$ In fact,
the proof of \cite[Proposition 3.2, p. 357]{MTW} shows that we may choose
$(\psi_n)_{n = 1}^\infty$ so that in addition $\psi_n\to \psi$ as $n\to \infty$ in $L^1_{\rm loc}(\Omega)$.  Thus,
$(|\psi_n|)\subset {\mathcal B}_{K}\cap C_0(\Omega) $ and $|\psi_n|\to |\psi|$ as $n\to \infty$ in $L^1_{\rm loc}(\Omega)$, hence in $\Dd$ and ${\mathcal B}_{K}^*.$
Let $u_n$ the solution of \eqref{intequOmega} with initial data $\psi_n\in C_0(\Omega)$ and let $v_n$ be the solution of \eqref{intequOmega} with initial
data $|\psi_n|\in C_0(\Omega).$ It is well known, by the $C_0(\Omega)$ local well-posedness  and comparison argument, that
$$|u_n(t)|\leq v_n(t),\; \forall\; t\in (0,T].$$ Letting $n\to \infty$ we get by Part (iii) that $u_n(t)\to u(t),\; v_n(t)\to v(t)$ in $C_0(\Omega)$ for $t\in (0,T]$ and hence
$$|u(t)|\leq v(t),\; \forall\; t\in (0,T].$$

(v) This is an immediate consequence of Part (iv).

(vi) The proof is analogous to the proof of Part (iv).
\end{proof}

\begin{rem}\label{error1}{\rm
There is a small error in the proof of
\cite[Proposition 2.5, p. 353]{MTW}.  It claims to show that if
$\psi \in \mathcal X$ then ${\rm e}^{t\Delta_\Omega}\psi\to
\psi$ a.e. as $t\to 0$.  The proof in fact shows that
${\rm e}^{t\Delta_\Omega}\psi\to
\psi$ in $L^1_{\rm loc}(\Omega)$ as $t\to 0$, which is indeed sufficient
to guarantee convergence in $\Dd$. The authors of \cite{MTW} regret this error.
}
\end{rem}

\begin{rem}{\rm By well-posedness of the integral equation in  $C_0(\Omega),$ the solution $u$ of  \eqref{intequOmega} with initial data $\psi\in \mathcal{B}_K,$ constructed in
Theorem \ref{th4} above, can be continued to a maximal solution $u \in C\left((0,T_{\max}(\psi)); C_0(\Omega)\right).$
If $T_{\max}(\psi)<\infty$ then $\lim_{t\to T_{\max}(\psi)}\|u(t)\|_\infty=\infty$.
}
\end{rem}

We have also obtained the following result.
\begin{To}
\label{globalproperties}
Assume  the hypotheses of Theorem \ref{th4}. Let $\psi\in  \mathcal{B}_K$ and let
$u \in C\left((0,T_{\max}(\psi)); C_0(\Omega)\right)$ be the solution of  \eqref{intequOmega} with initial data $\psi$ constructed by Theorem \ref{th4},
extended to its maximal existence time. Then we have the following.
\begin{itemize}
 \item[(i)] If $a=-1$ then $T_{\max}(\psi)=\infty.$
 \item[(ii)] If $a = 1$, then there exists $c > 0$ such that
 \begin{equation}
 \label{generallifespan}
\lambda^{[({1\over \alpha}-{\gamma+m\over 2})^{-1}]}T_{\max}(\lambda \psi) > c
 \end{equation}
 for all $\lambda > 0$.
 \item[(iii)] If $a = 1$, then $T_{\max}(\psi_0) < \infty$.
 Moreover,
 \begin{equation}
 \label{generallifespan2}
\lambda^{[({1\over \alpha}-{\gamma+m\over 2})^{-1}]}T_{\max}(\lambda \psi_0) = T_{\max}(\psi_0)
 \end{equation}
 for all $\lambda > 0$.
 \end{itemize}
\end{To}

\begin{proof}[Proof of Theorem \ref{globalproperties}]
(i) If $a = -1$, we claim that
\begin{equation}
\label{Kato}
|u(t)| \le {\rm e}^{t\Delta_\Omega}|\psi| \mbox{ for all } 0 < t < T_{\max}(\psi).
\end{equation}
To prove this, let $0 < t_0 < T$, where $T$ is as in the fixed point argument.
Since $u(t_0)\in C_0(\Omega)$, we know (by Kato's parabolic inequality)
that $|u(t + t_0)| \le {\rm e}^{t\Delta_\Omega}|u(t_0)|$ for all $0 < t < T_{\max}(\psi) - t_0.$  We next let $t_0 \to 0$, so that $u(t_0) \to \psi$, and likewise
$|u(t_0)| \to |\psi|$, in
$L^1_{\rm loc}(\Omega)$ by Theorem~\ref{th4}(i), and therefore in ${\mathcal B}_{K}^*$.
Finally, by \cite[Proposition 4.1, p. 359]{MTW}, it follows that
${\rm e}^{t\Delta_\Omega}|u(t_0)| \to {\rm e}^{t\Delta_\Omega}|\psi|$
in $C_0(\Omega)$ as $t \to 0$, which establishes \eqref{Kato}.
Global existence now follows since  ${\rm e}^{t\Delta_\Omega}|\psi| \in C_0(\Omega)$ for all $t > 0$.

(ii) It is clear that $T_{\max}(\psi)$
is bigger than the $T$ obtained from the fixed point argument.
Hence, by condition \eqref{e21n}, which implies \eqref{e22n},
we must have that
$$
K+2(\alpha+1)M^{\alpha+1}\left(1-{\gamma+m\over 2}\alpha\right)^{-1}\|e^{\Delta_\Omega} \psi_0\|_\infty^\alpha T_{\max}(\psi)^{1-{\gamma+m\over 2}\alpha}> M.
$$
for all $M > K$, where we have used \eqref{triangle} to simplify
$\int_0^T \|\Psi(\sigma)\|_{L^\infty(\Omega)}^\alpha
d\sigma$.  Likewise, it must be that
$$
\lambda K+2(\alpha+1)M^{\alpha+1}\left(1-{\gamma+m\over 2}\alpha\right)^{-1}\|e^{\Delta_\Omega} \psi_0\|_\infty^\alpha T_{\max}(\lambda \psi)^{1-{\gamma+m\over 2}\alpha}> M.
$$
for all $M > \lambda K$, for all $\lambda > 0$.
If we set $M = 2\lambda K$, we obtain the desired relationship.

(iii) We first show that $T_{\max}(\psi_0) < \infty$.
Let $u$ be the solution of  \eqref{NLHint} with initial data $\psi_0$,
extended to its maximal solution in $C_0(\Omega)$. Since $\psi_0\geq 0$,
 $u$  is positive on $\Omega$. By \cite[Theorem 1, p. 546]{W4}, we have that
\begin{equation}
\label{ineqinitia2}\|{\rm e}^{t\Delta_\Omega}\psi_0\|_\infty\leq \left(\alpha t\right)^{-1/\alpha} ,\; t\in (0,T_{\max}(\psi_0)).
\end{equation}
Since $\alpha < 2/(\gamma + m)$, formula \eqref{triangle} implies that $T_{\max}(\psi_0) < \infty$.

The formula for $T_{\max}(\lambda \psi_0)$ follows by a
rescaling argument, using the fact that $\psi_0$ is homogeneous of degree
$-(\gamma + m)$.  More precisely, we know that
\begin{equation*}
\lambda \psi_0 (x)= \mu^{2/\alpha}\psi_0 (\mu x),
\end{equation*}
where $\lambda = \mu^{\frac{2}{\alpha} - (\gamma + m)}$.
Furthermore, if $u(t,x)$ is the maximal solution of \eqref{intequOmega}
with initial value $\psi_0$, then $\mu^{2/\alpha}u(\mu^2 t, \mu x)$
is the maximal solution with initial value $\mu^{2/\alpha}\psi_0 (\mu \cdot)$,
so that
\begin{equation}
T_{\max}(\lambda \psi_0) = T_{\max}(\mu^{2/\alpha}\psi_0 (\mu \cdot)) = \mu^{-2} T_{\max}(\psi_0) = \lambda^{-[({1\over \alpha}-{\gamma+m\over 2})^{-1}]}T_{\max}( \psi_0).
\end{equation}
\end{proof}

\begin{rem}
{\rm It is well-known that a fixed point argument used to prove
local existence to a nonlinear integral equation can yield lower
estimates for the blowup rate.  The idea is to observe that the fixed point argument can not work on the interval
$[0, T_{\max} - t_0]$ with initial value $u(t_0)$.  To our knowledge,
this idea was first used in \cite[Section 4 and Remark 6(2)]{W2}
and then formalized more clearly in \cite[Proposition 5.3, p. 901]{MW}.
We are unable to apply this method in the present context since
it is not clear that the solution with initial data in $\X$ will remain
in $\X$ throughout its entire $C_0(\Omega)$ trajectory.  On the other hand, as the proof
of Theorem~\ref{globalproperties}(ii) shows, a similar idea can be used
to obtain lower estimates for $T_{\max} (\lambda u_0)$,
for all $\lambda > 0$.  The idea is to observe that $T_{\max} (\lambda u_0)$
is larger than the $T$ obtained by the fixed point argument, and then
apply the fixed point argument to initial values which are multiples $\lambda u_0$ of a given function $u_0$.  This idea is developed more fully
in \cite{TW4}.

}
\end{rem}

We have the following proposition showing in particular that the local existence result of Theorem \ref{th3} is optimal. We also give initial data which give rise to global solutions for $\alpha>2/(\gamma+m)$.

\begin{prop}
\label{optimalxrho} Let the positive integer $m$ and the real numbers
$\alpha,\; \gamma$ be such that $$1\leq m \leq N,\; 0<\gamma<N,\; \alpha>2/(\gamma+m) .$$
Then the following holds.
\begin{itemize}
\item[(i)] Assume $a=1.$ Let $\psi\in  \X$ such that $\psi\geq 0$ and  there exist $c>0$ and $L>0$ such that
 $$\psi(x)\geq c\psi_0(x),\; \forall\; x\in \Omega,\; |x|\leq L.$$ Then there is no local nonnegative solution of \eqref{intequOmega} on $\Omega$ with initial data $\psi.$
\item[(ii)] Assume $a=\pm 1.$ Let $\psi\in \X.$ Assume that there exists $t_0>0$ such that $$|\psi|\leq {\rm e}^{t_0\Delta_\Omega}\psi_0.$$
 Then for $\lambda>0$ sufficiently small there exists a global solutions $u$ of \eqref{intequOmega} on $\Omega$ with initial data $\lambda\psi.$
 Moreover, there exists $M>0,$ such that
 $$|u(t,x)|\leq M {\rm e}^{(t+t_0)\Delta_\Omega}\psi_0(x),\; \forall\; t>0,\; \forall\; x\in \Omega.$$
 \end{itemize}
\end{prop}

\begin{rems}$\;${\rm

\begin{itemize}
\item[(i)]The case $m=0$ of (i) in the previous proposition is known, see  \cite[Theorem 15.2, p. 76 and pp. 85-86]{QS}.
\item[(ii)] In \cite{LM,M1,M2,K} global solutions of \eqref{intequOmega} are proved to exist in the case $\alpha>2/(N+m).$  The initial data giving global existence, considered in these papers, are dominated by
Gaussian functions. In the part (ii) of the previous proposition, the initial data may have polynomial decay.
\item[(iii)] It is known that if $\alpha>2/N$ and for $\psi\in C_0(\R^N)$, and
\begin{equation}
\label{sigmest}
 |\psi(x)|\leq C(1+|x|^2)^{-\sigma/2},\; \sigma>2/\alpha
 \end{equation}
 with $C$ is sufficiently small,
 then the solution of   \eqref{NLHint} with initial value $u_0 = \psi$ is global. See \cite[Theorem 3, p. 32]{W2}.  Indeed, $\psi \in L^{q_c}(\R^N)$ with $q_c = N\alpha/2 > 1$.
By the previous proposition part (ii), using \cite[Lemma 2.6, p. 355]{MTW}, the initial value $\psi$ satisfies \eqref{sigmest} with $\sigma=\gamma+m>2/\alpha.$
But here we do not impose the condition $\alpha>2/N.$ In other words, this proposition allows the possibility that $\psi \in L^{q_c}(\R^N)$ with $q_c = N\alpha/2 \le 1$.
\end{itemize}}

\end{rems}
\begin{proof}[Proof of Proposition \ref{optimalxrho}]
(i) Assume that a local nonnegative solution  $u$ of  \eqref{NLHint} exists on $(0,T]$ with initial data $\psi$. Then  we have
\begin{eqnarray*}
u(t)& =& {\rm e}^{t\Delta_\Omega}\psi + \int_0^t {\rm
e}^{(t-\sigma)\Delta_\Omega} \big(u(\sigma)^{\alpha+1}\big)
d\sigma
\end{eqnarray*}
on $(0,T]\times\Omega.$
As in the proof of Theorem \ref{globalproperties} (iii), by \cite[Theorem 1, p. 546]{W4}, we have that
\begin{equation}
\label{ineqinitial}
{\rm e}^{t\Delta_\Omega}\psi\leq \left(\alpha t\right)^{-1/\alpha} ,\; \mbox{ on } \; (0,T]\times\Omega,
\end{equation}
for all $T<T_{\max}(\psi).$
 Write
$$\psi=\psi 1_{\{|x|\leq L\}}+\psi 1_{\{|x|> L\}}.$$  Then
\begin{equation}
\label{inqfwbis}
t^{{\gamma+m\over 2}}\|{\rm e}^{t\Delta_\Omega}\psi\|_\infty\geq t^{{\gamma+m\over 2}}\|{\rm e}^{t\Delta_\Omega}(\psi 1_{\{|x|\leq L\}})\|_\infty
-t^{{\gamma+m\over 2}}\|{\rm e}^{t\Delta_\Omega}(\psi 1_{\{|x|> L\}})\|_\infty.
\end{equation}
Using the relation \eqref{etoile}, we get first that
\begin{eqnarray}
\nonumber t^{{\gamma+m\over 2}}\|{\rm e}^{t\Delta_\Omega}(\psi 1_{\{|x|\leq L\}})\|_\infty&\geq &
ct^{{\gamma+m\over 2}}\|{\rm e}^{t\Delta_\Omega}(\psi_0 1_{\{|x|\leq L\}})\|_\infty\\ \nonumber &= &
ct^{{\gamma+m\over 2}}\|{\rm e}^{t\Delta_\Omega}(\psi_0 1_{\{|x|\leq L\}})(\sqrt{t}\;\cdot)\|_\infty\\&& \hspace{-2cm} \label{etoile4} =
c\|{\rm e}^{\Delta_\Omega}(\psi_0 1_{\{|x|\leq {L\over \sqrt{t}}\}})\|_\infty\to c\|{\rm e}^{\Delta_\Omega}\psi_0 \|_\infty\;
\mbox{ as }\; t\to 0,
\end{eqnarray}
and second, we have
\begin{eqnarray}
\nonumber t^{{\gamma+m\over 2}}\|{\rm e}^{t\Delta_\Omega}(\psi 1_{\{|x|> L\}})\|_\infty&\leq &
\|\psi\|_\X t^{{\gamma+m\over 2}}\|{\rm e}^{t\Delta_\Omega}(\psi_0 1_{\{|x|> L\}})\|_\infty\\ \nonumber &= &
\|\psi\|_\X t^{{\gamma+m\over 2}}\|{\rm e}^{t\Delta_\Omega}(\psi_0 1_{\{|x|> L\}})(\sqrt{t}\;\cdot)\|_\infty\\ \label{triangle3} &= &
\|\psi\|_\X \|{\rm e}^{\Delta_\Omega}(\psi_0 1_{\{|x|> {L\over \sqrt{t}}\}})\|_\infty\to 0\;
\mbox{ as }\; t\to 0.
\end{eqnarray}
The  convergence results follow by \cite[Proposition 4.1 (ii), p. 359]{MTW}. Then for $t>0$
 sufficiently small, by \eqref{inqfwbis}, \eqref{etoile4} and \eqref{triangle3}, there exists $C>0$ a constant such that
 \begin{equation}
 \label{cercle2}
 \|{\rm e}^{t\Delta_\Omega}\psi\|_\infty\geq Ct^{-{\gamma+m\over 2}}.
 \end{equation}
By \eqref{ineqinitia2} and \eqref{cercle2}, we have
$$
0<C\leq t^{(\gamma+m)/2}\|{\rm e}^{t\Delta_\Omega}\psi\|_\infty \leq \left(\alpha  t\right)^{-1/\alpha}t^{(\gamma+m)/2},\; t\in (0,T].
$$
This leads to a contradiction for $t$ sufficiently small and $\alpha>2/(\gamma+m).$ Then there is no local nonnegative solution.

(ii) Let $t_0>0$ and $\psi\in \X$ be as in the statement of the proposition. Let $$\Psi_{t_0}(t)=\Psi(t_0+t),\; t>0.$$
Then by \eqref{triangle}, we have
$$\|\Psi_{t_0}(t)\|_\infty=\|\Psi(t_0+t)\|_\infty= \|e^{\Delta} \psi_0\|_\infty (t_0+t)^{-(\gamma + m)/2},\; \forall \; t>0.$$
Hence, for any $\alpha>2/(\gamma+m),$
\begin{equation}
\label{2.3deTW-Ibis}\int_0^\infty \left\|\Psi_{t_0}(t)\right\|_\infty^\alpha dt<\infty.
\end{equation}
That is, \eqref{2.3deTW-I} is satisfied by $\Psi_{t_0}$ with $A=\infty,$  in place of $\Psi.$ We will apply Theorem \ref{th4} with $\Psi$ replaced
by $\Psi_{t_0}$ and $T=\infty.$ Let $K > 0$,  $M>0$ be such
that
$$ K+2(\alpha+1)M^{\alpha+1}\int_0^\infty \|\Psi_{t_0}(\sigma)\|_\infty^\alpha d\sigma\leq
M
$$
and
$$ 2(\alpha+1)M^{\alpha}\int_0^\infty \|\Psi_{t_0}(\sigma)\|_\infty^\alpha
d\sigma<1.$$
Let $\lambda>0$ be such that $\|\lambda\psi\|_\X\leq K.$ Then by Theorem \ref{th4}, there exists a unique solution  $u\in X_{\infty,\Psi_{t_0}}$ of
\eqref{intequOmega} on $(0,\infty)$ such that
$$\||u|\|_{X_{\infty,\Psi_{t_0}}}\leq M.$$ This proves the existence of global solution of \eqref{intequOmega} satisfying  the estimate
$$|u(t)|\leq M \Psi_{t_0}(t),\; \forall \, t>0,$$ on $\Omega$. This finishes the proof of the proposition.
\end{proof}

\section{well-posedness on the whole space}
\label{Wmgamma}
\setcounter{equation}{0}
In this section we use the results of the previous
section and an anti-symmetric  reflection argument to construct a solution of the nonlinear heat equation \eqref{NLHint} on $\R^N$ with singular initial values. In particular, we prove Theorems \ref{th1}, \ref{th3}  and \ref{th3bb}.
As in \cite{TW-I}, for $1 \le i \le N,$ let $T_i$ be the
operator defined on the space of functions $f$ by
$$[T_if](x_1,\,\cdots,\,x_{i-1},\,x_i,\,x_{i+1},
\,\cdots,\,x_N) =
f(x_1,\,\cdots,\,x_{i-1},\,-x_i,\,x_{i+1},\,\cdots,\, x_N).$$ A function  $f$ is anti-symmetric in $x_1,\; \cdots, x_m$ if it satisfies
\begin{equation}
\label{sym} T_{1}f = T_{2}f = \cdots = T_{m}f = -f.
\end{equation}
We denote the space of functions  which are  anti-symmetric in $x_1,\; \cdots, x_m$ by
\begin{equation}
\label{defA}
{\mathcal A} :=\left\{f : \R^N \rightarrow \R; \; f \mbox{  satisfies  }\; \eqref{sym}\right\}.
\end{equation}

\begin{defini}
Let $\Omega \subset \R^N$ be the domain given by \eqref{dmn}, and suppose $g : \overline \Omega \to \R$ is such that $g_{|\partial\Omega} \equiv 0$.
We denote by $\widetilde{g}: \R^N \rightarrow \R$ the anti-symmetric extension of $g$, i.e.
$\widetilde{g} \in {\mathcal A}$ and $\widetilde{g}_{|\Omega} = g$.
\end{defini}

If $g\in C_0(\Omega),$
 then $\widetilde{g}\in C_0(\R^N)$. It is proved in \cite[formula (3.6), p. 514]{TW-I} that
 \begin{equation}
 \label{heatsgrel}
 {\rm e}^{t\Delta_\Omega}g=\left({\rm e}^{t\Delta}\widetilde{g}\right)_{|\overline{\Omega}},\; \forall\; t>0,\;\forall\; g\in C_0(\R^N).
\end{equation}
Moreover, the commutation relation \eqref{commutation} is valid
with ${\rm e}^{t\Delta_\Omega}$ instead of $ {\rm e}^{t\Delta}$, and in particular
\begin{equation}
\label{commOmega}
D_{\sqrt t}{\rm e}^{t\Delta_\Omega} = {\rm e}^{\Delta_\Omega}D_{\sqrt t}
\end{equation}
for all $t > 0$, where the dilation operators $D_\lambda$ are given by \eqref{dilation}.

The following proposition is immediate.

\begin{prop}
\label{inteqrel}
Let $g \in C_0(\Omega)$ and let $v:[0, T_{\max}(g)) \to C_0(\Omega)$
be the maximal solution of \eqref{intequOmega} with initial value $v_0 = g$.
It follows that $\widetilde v :[0, T_{\max}(g)) \to C_0(\R^N) \cap {\mathcal A}$ is the maximal solution of \eqref{NLHint}  with initial value
$u_0 = \widetilde g \in C_0(\R^N)$.  In particular, $T_{\max}(g) =  T_{\max}(\widetilde g)$.

Similarly, let $f\in C_0(\R^N) \cap {\mathcal A}$ and let
$u :[0, T_{\max}(f)) \to C_0(\R^N)$ be the maximal solution of
\eqref{NLHint} with initial value $u_0 = f$.  It follows that $u(t) \in {\mathcal A}$
for all $0 \le t < T_{\max}(f)$ and that
$u_{|\overline{\Omega}}:[0, T_{\max}(f)) \to C_0(\Omega)$
is the maximal solution of \eqref{intequOmega} with initial value
 $v_0 = f_{|\overline{\Omega}} \in C_0(\Omega)$. In particular,
 $T_{\max}(f) =  T_{\max}(f_{|\overline{\Omega}})$.
\end{prop}

The analogous result is clearly true in other spaces where the
two integral equations are well-defined.  The point of this paper
is to use the solutions on $\Omega$ to construct solutions
on $\R^N$, which is not quite the same as the above proposition.

If $\psi\in \X$, given by \eqref{spc}, it is proved in \cite[Definition 1.6, p. 346 and Proposition 5.1, p. 361]{MTW} that the pointwise
 anti-symmetric extension $\widetilde{\psi}$ of $\psi$ to $\R^N$ has a natural interpretation as a tempered distribution. It is natural to define
\begin{equation}
\label{extention}
{\widetilde{\X}}=\left\{\varphi\in L^1_{\mbox{loc}}\left(\R^N\setminus\{0\}\right)\cap {\mathcal A} ,\;  \varphi|_{\Omega}\in {\mathcal X}\right\}=\{\widetilde{\psi}\; |\; \psi\in \X\} \subset \Sd.
\end{equation}
and
\begin{equation}
\label{extentionK}
{\widetilde{{\mathcal B}_K}}:=\{{\widetilde{\psi}}\;|\; \psi\in {\mathcal B}_K\},
\end{equation}
where ${\mathcal B}_K, K > 0$, is given by \eqref{ball}. Also ${\widetilde{{\mathcal B}_K^*}}$ inherits the metric space
structure from  ${\mathcal B}_K^*$ (recall Definition~\ref{BKstar}). If $m = 0$, then $\widetilde{\X} = \X$ and
 $\widetilde{{\mathcal B}_K} ={\mathcal B}_K$.

We recall the following result from \cite{MTW} and \cite{CDW-DCDS} showing the equivalence of various kinds of convergence.

\begin{prop}[Propositions 3.1(i) and 5.1 in \cite{MTW} and Proposition 2.1(i) in \cite{CDW-DCDS}]
\label{equivlim}
Let $0 < \gamma < N$ and $m$ an integer with $0 \le m \le N$,
and let $(\psi_k)_{k = 1}^\infty \subset {\mathcal B}_K$ and $\psi\in {\mathcal B}_K$. The following are equivalent:
 \begin{itemize}
  \item[(i)] $\psi_k\to \psi$ in ${\mathcal B}_K^*$ as $k\to \infty$ ;
  \item[(ii)] $\psi_k\to \psi$ in $\Dd$ as $k\to \infty$ ;
  \item[(iii)] ${\widetilde{\psi_k}}\to {\widetilde{\psi}}$ in $\Sd$ as $k\to \infty$.
 \end{itemize}
\end{prop}

 \begin{rem}
 
 \label{equivlim2} {\rm
If $(\psi_k)_{k = 1}^\infty \subset {\mathcal B}_K$, $\psi\in L^1_{\rm loc}(\Omega)$, and $\psi_k\to \psi$ in $\Dd$ as $k\to \infty$, then in fact 
$\psi\in {\mathcal B}_K$ and the above conditions hold.  This follows either from the
compactness of ${\mathcal B}_K^*$ ( so that the $(\psi_k)_{k = 1}^\infty$ must have
a convergent subsequence in ${\mathcal B}_K^*$), or from the fact that convergence in $\Dd$
preserves a uniform bound.
}
\end{rem}

Moreover, (\cite[Proposition 5.1]{MTW}),
 $$\widetilde{{\rm e}^{t\Delta_\Omega}\psi}= {\rm e}^{t\Delta}\widetilde{\psi},\; \forall\; t>0,\;\forall\;\psi\in \X.$$
For $0<\gamma<N,$ we put
\begin{equation}  \label{phi0}%
\varphi_0 = (-1)^m \partial_1 \partial_2 \cdots \partial_m
\left( \vert \cdot \vert^{-\gamma} \right) \in \Sd.
\end{equation}
By \cite[Proposition 2.1 p. 347]{MTW}, it follows that
 \begin{equation}
 \label{psi0isphi0}
 \widetilde \psi_0 =\varphi_0,
  \end{equation}
  in $\Sd$.

 By \eqref{phi0}, if follows that
\begin{eqnarray}  \label{relphipsi}%
\Phi(t)=e^{t\Delta} \varphi_0 &=& G_t\star \varphi_0
 = \Phi_0(t) \star  |\cdot|^{-\gamma}\\
 \label{relphipsi1}
 &=& (-1)^m \partial_1 \partial_2 \cdots \partial_m
\left(e^{t\Delta}|\cdot|^{-\gamma}\right),
\end{eqnarray}
in $\Sd,$ where  $\Phi_0$ is given by
\begin{equation} \label{ksi0}%
\Phi_0 (t) =(-1)^m \partial_1 \partial_2 \cdots \partial_m G_t,\
t>0 .
\end{equation}
We have that $\Phi : (0,\infty) \to C_0(\R^N)$ is a solution of
the linear heat equation such that $T_{i}\Phi(t) = -\Phi(t)$ for all
$1 \le i \le m$. The key observation, to use the results of the previous section,  is that, by \cite[Proposition 2.1, p. 347 ]{MTW}, we have
$$\Phi(t)_{|{\overline{\Omega}}}=\left(e^{t\Delta} \varphi_0\right)_{|{\overline{\Omega}}}={\rm e}^{t\Delta_\Omega}\psi_0=\Psi(t),\; \forall\; t>0.$$
Also $$\widetilde{\Psi}(t)=\Phi(t),\; \forall\; t>0,$$
where $\widetilde{\Psi}(t)$
is the unique anti-symmetric extension of $\Psi(t)$ to $\R^N.$ In particular, we have $$\|\Phi(t)\|_{L^\infty(\R^N)}=\|\Psi(t)\|_\infty,\; \forall\; t>0.$$

We now wish to re-formulate Theorem ~\ref{th4} in terms of a result on the whole space $\R^N$. This reformulation is possible thanks to  \cite[Proposition 3.3, p. 515]{TW-I}, which is analogous to Proposition~\ref{inteqrel}. The following theorem is thus an immediate consequence of
Theorem ~\ref{th4} and Proposition~\ref{equivlim}. Of course in the case $m = 0$,
the following result is the same as Theorem ~\ref{th4}.

\begin{To}[well-posedness in ${\widetilde{\X}}$]
\label{th3'}
Let the positive integer $m$ and the real numbers
$\alpha,\, \gamma$ be such that $$ 0\leq m \leq N,\; 0<\gamma<N, \; 0<\alpha<{2\over \gamma+m}.$$ Let $K,\; M,\; T>0$
be satisfying \eqref{e21n} and \eqref{e22n}. Then for every $u_0\in {\widetilde{\X}}$ with $||u_0||_{{\widetilde{\X}}} \le K$, there exists  a solution
$u \in C\left((0,T]; C_0(\R^N)\cap {\mathcal A}\right)$ of \eqref{NLHint} such that $u(t)\to u_0$ in $\Sd$ as $t\to 0.$ Furthermore,
\begin{itemize}
\item[(i)] $|u(t)|\leq M e^{t\Delta} \psi_0,\; \forall\; t\in [0,T],\; \forall \; x\in \Omega.$
\item[(ii)] If $v \in C\left((0,T]; C_0(\R^N)\cap {\mathcal A}\right)$ is a solution of \eqref{NLHint} such that $|v(t)|\leq M e^{t\Delta} \psi_0,\; \forall\; t\in [0,T],\; \forall \; x\in \Omega,$
with $v(0)=u_0$ then $u(t)=v(t),\; \forall\; t\in [0,T].$
\item[(iii)] Let $(u_{0,n})\subset {\widetilde{{\mathcal B}_K}},\; u_0\in {\widetilde{{\mathcal B}_K}}$ with $u_{0,n}\to u_0$ in $\Sd.$
Let $u_n$ be the solution of (\ref{NLHint}) with initial data $u_{0,n}$
and
$u$ be the solution of (\ref{NLHint}) with initial data $u_0$. Then $u_{n}(t)\to u(t)$
in $C_0(\R^N),$ $\forall\; t\in (0,T],$ and uniformly in $[t_0,T],\; \forall \; t_0\in (0,T).$
\end{itemize}
Moreover $u$ can be continued to a maximal
solution $u \in C\left((0,T_{\max}); C_0(\R^N)\right)$ and
$u(t) \in  {\mathcal A}$ for all $0 < t <T_{\max}$.

\end{To}

We now give the proofs of Theorems \ref{th1}, \ref{th3} and \ref{th3bb}.

\begin{proof}[Proof of Theorem~\ref{th1}.] The proof  is the same as in \cite[Theorem 1.1, p. 506]{TW-I} which is valid also for minus sign in front of the nonlinearity.
The last statement in the theorem follows by the fact that in the absorption case the solutions are global. This completes the proof of the theorem.
\end{proof}

\begin{proof}[Proof of Theorem~\ref{th3}]  Theorem ~\ref{th3} is a particular case of Theorem \ref{th3'} with $u_0=K\varphi_0.$ Here we fix $K$ and we take $M,\; T$
such that \eqref{e21n} and \eqref{e22n} are satisfied. Hence the proof of the existence
follows. The others statements follows by Theorem \ref{globalproperties}.
\end{proof}

\begin{proof}[Proof of Theorem~\ref{th3bb}]  The proof follows by Theorem ~\ref{th3'} with $u_0\in {\widetilde{\X}}$ and by  Theorem \ref{globalproperties}.
\end{proof}

\section{Blow-up results}
\label{blowup-span}
\setcounter{equation}{0}
This section is devoted to the proofs of Theorems \ref{blowup}, \ref{lifespanderiveedelta}, \ref{Dickstein-type}, \ref{Dickstein-type2} and Corollary \ref{nonexistlimit}.
Let $K>0.$ For $\psi\in {\mathcal B}_K$, given by \eqref{ball}, where
$0 < \gamma < N$ and $m$ an integer with $0 \le m \le N$,  we recall from \cite[Formula (1.18), p. 344]{MTW}
for $1\leq m\leq N$ and
\cite[Formula (1.17), p. 1107]{CDW-DCDS} for $m=0$, the following definition
\begin{equation}
\label{Zu0}
{\mathcal{Z}}(\psi):=\left\{z\in {\mathcal{B}_K},\; \exists\, \lambda_n\to \infty, n\to \infty,\; \mbox{ such that }\;
\lim_{n\to \infty}\lambda_n^{\gamma+m}D_{\lambda_n}\psi=z\, \mbox{in}\, {\mathcal{B}}_K^*\right\},
\end{equation}
where ${\mathcal{B}}_K^*$ denotes ${\mathcal B}_K$ endowed with respect to the weak$^*$ topology on $\X$ and the dilation operator $D_\lambda$ is given by \eqref{dilation}.
 By \cite[Proposition 3.1 (iii), p. 356]{MTW} and \cite[Remark 2.4(iii), p. 1112]{CDW-DCDS}, ${\mathcal{Z}}(\psi)$ is a nonempty, compact,
and connected subset of the compact metric space ${\mathcal B}_K^*.$

Also, it is important to remark that the conditions on $z \in L^1_{\rm{loc}}\left(\R^N\setminus\{0\}\right)$ in the statements of Theorems~\ref{blowup} and \ref{Dickstein-type2} can be interpreted to say that
$z_{|\overline\Omega} \in {\mathcal{Z}}(f_{|\overline\Omega})$.

\begin{proof}[Proof of Theorem \ref{blowup}] This result concerns
solutions in $C_0(\R^N) \cap {\mathcal A}$, where ${\mathcal A}$ is defined by
\eqref{defA}, i.e. solutions which are anti-symmetric in $x_1, x_2, \dots , x_m$.  Moreover, the solutions in question are positive
on $\Omega$.
It is clear from Proposition~\ref{inteqrel}
that it suffices to consider the restrictions of these solutions to $\overline\Omega$,
as solutions of \eqref{intequOmega}, rather than as solutions on
$\R^N$ of \eqref{NLHint}.  This will enable us to use the positivity.

To simplify the notation, we will use, {\it by abuse of notation},
the same letters to denote functions on $\R^N$ and their restrictions
to $\overline\Omega$.  For example, if $f \in C_0(\R^N)$ is as in the statement of the theorem, and therefore in ${\mathcal A}$, in this proof $f$ will denote its restriction to $\overline\Omega$, as an element of $C_0(\Omega).$ Also, we consider the case $m \ge 1$, the case
$m = 0$ being entirely analogous.

 Let $u_0$ and $f$ be as in the statement of the theorem (restricted to $\overline\Omega$).  By comparison, it suffice to show that $T_{\max}(f)<\infty.$
 We argue by contradiction and assume that $T_{\max}(f)=\infty.$ By \cite[Theorem 1, p. 546]{W4}, since $f \ge 0$, it follows that
\begin{equation}
 \label{Criterum}
 \|e^{t\Delta_\Omega}f\|_{L^\infty(\Omega)}\leq \left(\alpha t\right)^{-1/\alpha},\; \forall\; t>0.
\end{equation}
Moreover, by \eqref{commOmega} we have
$$
 \|e^{t\Delta_\Omega}f\|_{L^\infty(\Omega)}= \|D_{\sqrt{t}}\left(e^{t\Delta_\Omega}f\right)\|_{L^\infty(\Omega)}=\|e^{\Delta_\Omega}D_{\sqrt{t}}f\|_{L^\infty(\Omega)}.
 $$
Hence,
\begin{equation}
 \label{property1m}
 t^{{\gamma+m\over 2}}\|e^{t\Delta_\Omega}f\|_{L^\infty(\Omega)}=\left\|e^{\Delta_\Omega}\left(t^{{\gamma+m\over 2}}D_{\sqrt{t}}f\right)\right\|_{L^\infty(\Omega)}.
\end{equation}

By the hypotheses of the theorem, $f\in C_0(\Omega)\cap \mathcal{B}_K,$ where $\mathcal{B}_K$ is defined by \eqref{ball}. Let $z \in L^1_{\rm{loc}}(\Omega)$, $z \not\equiv 0$, and $\lambda_n \to \infty$
 be such that
\begin{equation*}
%\label{subseqlim}
\lambda_n^{\gamma + m}D_{\lambda_n}f \to z
\end{equation*}
in ${\mathcal D'}(\Omega)$, as $n \to \infty$. We claim that  $z\in \mathcal{Z}(f)$ and  $\lambda_n^{\gamma + m}D_{\lambda_n}f \to z$ in ${\mathcal{B}}_K^*$,
where  $\mathcal{Z}$ is defined in \eqref{Zu0}. In fact, since ${\mathcal{B}}_K^*$ is compact, there exists a subsequence $(\lambda_{n_k})_{k=1}^{\infty}$ and
$\zeta$ in ${\mathcal{B}}_K^*$ such that $\lambda_{n_k}^{\gamma + m}D_{\lambda_{n_k}}f \to \zeta$ in ${\mathcal{B}}_K^*.$ But convergence in ${\mathcal{B}}_K^*$ implies convergence
in ${\mathcal D'}(\Omega)$, so $z=\zeta \in {\mathcal{B}}_K^*$.
It then follows from Proposition~\ref{equivlim} that
$\lambda_n^{\gamma+m} D_{\lambda_n} f\to z$ in $\mathcal{B}^*_{K}$ as $n\to \infty$, and so  $z\in \mathcal{Z}(f)$.

Let now $t_n=\lambda_n^2\to \infty,$ as $n\to \infty.$
Then $t_n^{{\gamma+m\over 2}}D_{\sqrt{t_n}}f\to z$
 in $\mathcal{B}^*_{M}$ as $n\to \infty.$  By \cite[Proposition 4.1(ii), p. 359]{MTW}, it follows that
 \begin{equation}
  \label{property2m}
e^{t\Delta_\Omega}\left(t_n^{{\gamma+m\over 2}}D_{\sqrt{t_n}}f\right) \to  e^{\Delta_\Omega}z,\;\; \mbox{ in }\; C_0(\Omega),
 \end{equation}
as $n\to \infty$, and so
 \begin{equation}
  \label{condition2m}\lim_{n\to\infty}t_n^{{\gamma+m\over 2}}\|e^{t_n\Delta_\Omega}f\|_{L^\infty(\Omega)}=\|e^{\Delta_\Omega}z\|_{L^\infty(\Omega)}\not=0.
  \end{equation}
 Using \eqref{Criterum}, we get
  \begin{equation}
  \label{reformulationCriterumm}
  t_n^{{\gamma+m\over 2}}\|e^{t_n\Delta_\Omega}f\|_{L^\infty(\Omega)}\leq \alpha^{-1/\alpha}t_n^{{\gamma+m\over 2}-{1\over \alpha}},\; \forall\; n>0.
  \end{equation}
By \eqref{condition2m}  and \eqref{reformulationCriterumm}, we see that for any $\varepsilon>0,$
$$\|e^{\Delta_\Omega}z\|_{L^\infty(\Omega)}-\varepsilon<\alpha^{-1/\alpha}t_n^{{\gamma+m\over 2}-{1\over \alpha}}\; \mbox{ for large n}.$$
Letting $n\to \infty$ in the last inequality, we get a contradiction if
$\alpha<2/(\gamma+m)$ with the fact that $\|e^{\Delta_\Omega}z\|_{L^\infty(\Omega)}> 0.$ We also get a contradiction if $\alpha=2/(\gamma+m)$ by letting $n\to \infty,$ since in this case
$\|e^{\Delta_\Omega}z\|_{L^\infty(\Omega)}>\alpha^{-1/\alpha}.$
It follows that  $T_{\max}(f)<\infty.$ This completes the proof  of Theorem \ref{blowup}.
\end{proof}

The following remark, inspired by \cite[Proposition 2.12, p. 1117]{CDW-DCDS}, gives some examples of functions $f$ satisfying the conditions of Theorem~\ref{blowup} such that ${\mathcal{Z}}(f_{|\overline\Omega})\not=\{0\}$,
i.e. for which an appropriate $z \in L^1_{\rm{loc}}\left(\R^N\setminus\{0\}\right) \cap {\mathcal A}$ exists.  In particular, it generalizes the example
discussed in Remark~\ref{compact}(iii).

Note first that if $\gamma+m<N$ and $f\in L^p(\Omega)$ for some $1\leq p<N/(\gamma+m)$ then
${\mathcal{Z}}(f)=\{0\}.$ This follows by scaling argument. See \cite[Proposition 2.5 (iii), p. 1112]{CDW-DCDS} for the case $m=0.$
Also, if $f \in C_0(\R^N) \cap {\mathcal A}$ has compact support,
then ${\mathcal{Z}}(f)=\{0\}.$

\begin{rem}
 \label{exampleinitialdata}
{\rm
 Let $0\leq m\leq N$ be an integer and $0<\gamma<N.$ Let $g\in C(\R,\R)$ be a bounded function such that $g(t)\ge 0$ for $t\in\R$. Let $\zeta\in C(S^{N-1})
 \cap {\mathcal A}$, and suppose that $f \in C_0(\R^N) \cap {\mathcal A}$ is such that
 $f_{|\overline\Omega} \ge 0$ and satisfies
 \begin{equation}
  \label{generalexample}
  f(x)=\psi_0(x)g\left(\log(|x|)\right)\zeta\left({x\over |x|}\right),\; x\in \Omega,
  |x| \ge \rho,
 \end{equation}
 where $\psi_0$ is given by \eqref{psi0} and ${\mathcal A}$ is given by \eqref{defA}.
 Then there exists $K>0$ such that $f_{|\overline\Omega}\in C_0(\Omega)\cap \mathcal{B}_K.$ Moreover, we have
 \begin{equation*}
\lambda^{(\gamma+m)}D_{\lambda}f(x)=\psi_0(x)g\left(\log(|x|)+\log \lambda\right)\zeta\left({x\over |x|}\right),\; x\in \Omega,
|x| \ge \rho/\lambda.
 \end{equation*}
It is now a simple matter to choose a wide variety of functions $g$ such that
${\mathcal{Z}}(f_{|\overline\Omega})\not=\{0\}$.  For example, if $g$ is periodic, then
the function $z$, given by expression in \eqref{generalexample}
is an element of ${\mathcal{Z}}(f)$.
}
\end{rem}

\begin{rem}{\rm
We wish to point out that the proof of Theorem~\ref{blowup}, which
gives a genuine improvement of known blowup results for equation
\eqref{NLheat}, including in the case $m = 0$, depends only on calculations with the (linear) heat
semigroup.  The first ingredient of the proof is \cite[Theorem 1, p. 546]{W4},
which gives a necessary condition on the heat semigroup for
existence of nonnegative solutions to \eqref{NLheat} on a certain
time interval.  The second ingredient is \cite[Proposition 4.1(ii), p. 359]{MTW},
and \cite[Proposition 3.8, p. 1123]{CDW-DCDS} in the case $m = 0$.
In particular, the example discussed in Remark~\ref{compact}(iii) could
have been observed after the appearance of \cite{W4} and \cite{CDW-DCDS}.
}
\end{rem}

Both Theorem~\ref{Dickstein-type} and Corollary~\ref{nonexistlimit} are consequences of Theorem~\ref{Dickstein-type2}.  The key arguments of the proof of  Theorem~\ref{Dickstein-type2} i.e. the rescaling and the continuity
of blowup times, are inspired by the paper \cite{D}.  Furthermore,
in order to apply these methods, we need the well-posedness of the
integral equation \eqref{NLHint} on the space
$\widetilde{\X}$ given by \eqref{extention}, which was shown
in Theorem~\ref{th3'}.  The following theorem is a slight reformulation and generalization
of  Theorem~\ref{Dickstein-type2}.  (Recall Proposition~\ref{equivlim} and 
Remark~\ref{equivlim2} above.)

\begin{To}
\label{th3b}
Let the integer $m$ and the real numbers
$\alpha,\; \gamma$ be such that $$0\leq m \leq N,\; 0<\gamma<N,\; 0<\alpha<{2\over \gamma+m},\; (N-2)\alpha<4.$$
Let $u_0\in {\widetilde{\X}}$ and let   $u \in C\left((0,T_{\max}(u_0)); C_0(\R^N) \cap {\mathcal A}\right)$ be the maximal  solution of \eqref{NLHint}
given by Theorem \ref{th3'}
such that $u(t)\to u_0$ in $\Sd$ as $t\to 0.$  Let $f\in \widetilde{{\mathcal{B}}_K},$ for some $K>0.$  For $\lambda>0$, let  $u^\lambda\in C\left((0,T_{\max}(\lambda f)); C_0(\R^N)\right)$ be the maximal solution of \eqref{NLHint} given by Theorem \ref{th3'}
such that $u(t)\to \lambda f$ in $\Sd$ as $t\to 0$. 

Suppose that there exists a sequence $(\tau_n)_{n = 1}^\infty$, with $\tau_n > 0$,
such that
\begin{equation}
\label{conditionf}
\tau_n^{-{\gamma+m\over 2}}D_{1\over \sqrt \tau_n}f\to u_0\; \mbox{ in } {\mathcal S}'(\R^N)\; \mbox{ as }\, n\to \infty.
\end{equation}
It follows that
\begin{equation}
\label{eqlimit}
\lim_{n\to \infty} \lambda_n^{[({1\over \alpha}-{\gamma+m\over 2})^{-1}]}T_{\max}(\lambda_n f)=T_{\max}(u_0),
\end{equation}
where $\lambda_n=\tau_n^{{1\over \alpha}-{\gamma+m\over 2}}$.
In particular, if $T_{\max}(u_0)<\infty$ then $T_{\max}(\lambda_n f)<\infty$ for all $n$ sufficiently large.

Furthermore, if $\tau_n \to 0$ as $n \to \infty$, then $\lambda_n \to 0$ as $n \to \infty$,
and if $\tau_n \to \infty$ as $n \to \infty$, then $\lambda_n \to \infty$ as $n \to \infty$.
\end{To}

\begin{proof}
By the standard
invariance properties of solutions to
(\ref{NLheat}), for any $\tau > 0$,
 \begin{equation}
 \label{standardinv}
 u_\tau^\lambda(t,x)=\tau^{-1/\alpha}u^\lambda\left({t \over \tau},{x \over \sqrt \tau}\right)
 \end{equation}
is the maximal
solution of (\ref{NLheat}) with initial value
$$
u_{0,\tau}^\lambda= \lambda \tau^{-1/\alpha}D_{1 \over \sqrt \tau}f
\in \widetilde{{\mathcal{B}}_{\rho K}}
$$
 in the sense of Theorem \ref{th3'}, where $\rho = \lambda\tau^{-({1 \over \alpha}-{\gamma+m \over 2})}$.
For  $\tau> 0$ we let $\lambda = \lambda(\tau)$ be given by
$$
\lambda =\lambda(\tau)= \tau^{{1 \over \alpha}-{\gamma+m \over 2}}.
$$
With this choice of $\lambda = \lambda(\tau)$, we refer to the solutions
$u_{\tau}^{\lambda(\tau)}$ as $u_\tau.$ It follows that the initial data of $u_\tau,$ is $u_{0,\tau} =
f_\tau$ where
\begin{equation}
\label{ftau}
f_\tau:=\tau^{-{\gamma+m \over 2}}D_{1 \over \sqrt \tau}f \in \widetilde{{\mathcal{B}}_{K}}.
\end{equation}
By \eqref{standardinv},
\begin{equation}
\label{scalingtmax}
T_{\max}(u_{0,\tau})=\tau T_{\max}(\lambda f)=\lambda^{[({1\over \alpha}-{\gamma+m\over 2})^{-1}]}T_{\max}(\lambda f).
\end{equation}
Moreover, $f_\tau, u_0 \in \widetilde{{\mathcal{B}}_K}$,
given by \eqref{extentionK}.

Since $u_{0,\tau_n} =
f_{\tau_n} \to u_0$ in $\Sd$, by \eqref{conditionf}, it follows from
Theorem \ref{th3'} that there exists $T >0$ such that $u_{\tau_n}(t)\to u(t)$
 in $C_0(\R^N)$  for all $t\in (0,T]$.  Fix $0 < \delta < T$ and
let $u_{\tau_n,\delta}(t)=u_{\tau_n}(t+\delta)$ and $u_\delta(t)=u(t+\delta)$ with maximal existence times respectively $T_{\max}(u_{0,\tau_n})-\delta$ and $T_{\max}(u_0)-\delta$. Since $(N-4)\alpha<2$, we know by \cite{GMS1,GMS2} that blowing up solutions are always type I. Hence, by continuity of the maximal time of existence in $C_0(\R^N)$ for type I blow-up solutions  (see \cite{FMZ}), in the case where $T_{\max}(u_0) < \infty$,
and by the lower semi-continuity of the
maximal time of existence in the case $T_{\max}(u_0) = \infty$, it follows that
\begin{equation*}
T_{\max}(u_{\tau_n}(\delta)) \to T_{\max}(u(\delta)),
\end{equation*}
which implies that
\begin{equation*}
T_{\max}(u_{0,\tau_n}) \to T_{\max}(u_0).
\end{equation*}
The relation \eqref{eqlimit} now follows from \eqref{scalingtmax}.

Finally, since $\alpha < 2/(\gamma+m),$ we have that as $\tau_n \to 0$ 
(respectively $\infty$) if and only if
$\lambda_n \to 0$ (respectively $\infty$).
\end{proof}

As mentioned above, Theorem~\ref{Dickstein-type2} is included in
Theorem~\ref{th3b}, and has thus been proved.  We next turn to
Theorem~\ref{Dickstein-type}.

\begin{proof}[Proof of Theorem \ref{Dickstein-type}] Let $f\in C_0(\R^N)\cap {\mathcal A}$ satisfying \eqref{boundf}. It follows that there exist two constants, $0 < c < C < \infty$ and $L > 0$, such that
$$
c\psi_0 \le f(x) \le C\psi_0, \quad \forall |x| \ge L, \, x \in \Omega,
$$
where $\psi_0$ is given by \eqref{psi0},
and so
$$
c\psi_0 \le f_\tau(x) \le C\psi_0, \quad \forall |x| \ge L\sqrt\tau, \, x \in \Omega,
$$
where $f_\tau$ is given by \eqref{ftau}.
Hence, if $u_0$ and $\tau_n$ satisfy the hypotheses of Theorem~\ref{th3b},
and if $\tau_n \to 0$,
we must have that
$$
c\psi_0 \le u_0(x) \le C\psi_0, \quad \forall x \in \Omega.
$$
By comparison, and Theorems~\ref{globalproperties} and \ref{th3b}, that
\begin{equation}
\label{cadre}
T_{\max}(C\psi_0) \le
T_{\max}(u_0) = \lim_{n\to \infty} \lambda_n^{[({1\over \alpha}-{\gamma+m\over 2})^{-1}]}T_{\max}(\lambda_n f) \le
T_{\max}(c\psi_0)
< \infty,
\end{equation}
where $\lambda_n = \tau_n^{({1\over \alpha}-{\gamma+m\over 2})} \to 0$
as $n \to \infty$.

It follows from \eqref{cadre} that
\begin{equation}
\label{cadre2}
T_{\max}(C\psi_0) \le
\liminf_{\lambda \to 0} \lambda^{[({1\over \alpha}-{\gamma+m\over 2})^{-1}]}T_{\max}(\lambda f) \le
\limsup_{\lambda \to 0} \lambda^{[({1\over \alpha}-{\gamma+m\over 2})^{-1}]}T_{\max}(\lambda f) \le
T_{\max}(c\psi_0)
\end{equation}
To see this, suppose to the contrary there is a sequence $\lambda_n \to 0$
such that $\lambda_n^{[({1\over \alpha}-{\gamma+m\over 2})^{-1}]}T_{\max}(\lambda_n f)$ does not ultimately fall in the range given by \eqref{cadre2}.
Passing to a subsequence, we may assume, by the compactness of
${\mathcal{Z}}(f)$, defined by \eqref{Zu0}, and by \cite[Proposition 5.1 p. 361]{MTW}, that there exists $u_0$
such that the hypotheses of the Theorem~\ref{th3b} are satisfied
with $\tau_n = \lambda_n^{[({1\over \alpha}-{\gamma+m\over 2})^{-1}]}$,
and hence we obtain \eqref{cadre}. This completes the proof of the theorem.
\end{proof}

We now give the proof of Corollary \ref{nonexistlimit}.

\begin{proof}[Proof of Corollary \ref{nonexistlimit}]
As in the proof of Theorem \ref{blowup}, it suffices to consider the
restrictions of functions to $\overline{\Omega}$.  By \cite[Proposition 2.9, p. 1115]{CDW-DCDS} and \cite[Section 3]{MTW}, given
$0 < c_1 < c_2 \le K < \infty$,
there exists $f \in {\mathcal B}_K^* \cap C_0(\Omega)$, such that
${\mathcal{Z}}(f) = \{c\psi_0 ;
c_1 \le c \le c_2\}$.
(A result analogous to \cite[Proposition 2.9, p. 1115]{CDW-DCDS} was
not explicitly given in \cite[Section 3]{MTW} but is easily proved.)
It follows that
$$
\liminf_{\lambda\to 0}\lambda^{[({1\over \alpha}-{\gamma+m\over 2})^{-1}]}T_{\max}(\lambda f) = T_{\max}(c_2\psi_0) < T_{\max}(c_1\psi_0)
= \limsup_{\lambda\to 0}\lambda^{[({1\over \alpha}-{\gamma+m\over 2})^{-1}]}T_{\max}(\lambda f),
$$
where we have used formulas \eqref{lifespanlim} and \eqref{generallifespan2}.
\end{proof}

\begin{rems}
\label{limitexist}$\;$
{\rm
\begin{itemize}
\item[(i)] Let $f \in C_0(\R^N) \cap {\mathcal A}$ be a universal solution in  as in \cite[Theorem 1.4, p. 345]{MTW}, or \cite[Theorem 1.2, p. 1108]{CDW-DCDS} in the case $m = 0$. More precisely, for some $K > 0$, $f_{|\Omega}:=g\in \mathcal{B}_K\cap C^\infty(\Omega)\cap C_0(\Omega)$ and
 $$
 {\mathcal{Z}}(g)= \mathcal{B}_K,
 $$
    where ${\mathcal{Z}}(g)$ is defined by \eqref{Zu0}.
Hence ${\mathcal{Z}}(f)=\widetilde{{\mathcal{Z}}(g)}=\widetilde{{\mathcal{B}}_K}.$
It follows that for every $z \in {\widetilde{{\mathcal B}_K}}$ there exists
a sequence $\mu_n \to \infty$ such that \eqref{subseqlim2} holds.
Thus, assuming the conditions on $\alpha$ in Theorem~\ref{Dickstein-type2},
it follows that
$$
\lambda_n^{[({1\over \alpha}-{\gamma + m\over 2})^{-1}]}T_{\max}(\lambda_n f) \to T_{\max}(z)
$$
where $\lambda_n = \mu_n^{-[{2\over \alpha}-(\gamma + m)]} \to 0$.

This adds to the examples described in the previous remarks.

\item[(ii)] If the initial data $f$ in Theorem \ref{Dickstein-type} satisfies also
$$\lim_{|x|\to \infty,\; x\in \Omega}{f(x)\over \psi_0(x)}=c>0,$$
Then
$$
{\mathcal{Z}}(f_{|\Omega})=\left\{c\psi_0\right\}.
$$
Hence
$$
\lim_{\lambda\searrow 0} \lambda^{\left({1\over \alpha}-{\gamma+m\over 2}\right)^{-1}}T_{\max}(\lambda f)=T_{\max}(c\psi_0)<\infty.
$$
In particular, the last limit exists and is finite. The case  $m=0$  is known, see \cite[Theorem 1.3, p. 307]{D}. See also \cite[Theorem 2, p. 172]{GW} for the case $m=0$,  but only for a more restrictive class of initial data.
\end{itemize}}
\end{rems}

The following two results are analogues of Theorem~\ref{Dickstein-type}
and Corollary~\ref{nonexistlimit} where the asymptotic behavior
of $T_{\max}(\lambda f)$ as $\lambda \to \infty$ is studied, as opposed
to $\lambda \to 0$. 

\begin{prop}
\label{nonDickstein-type} Let $a=1$ in \eqref{NLHint} and let the  integer $m$ and the real numbers $\alpha,\; \gamma$ be such that
$$
0\leq m\leq N,\;  0<\gamma <N,\; 0<\alpha<{2\over \gamma+m},\; (N-2)\alpha<4.
$$
Let $f\in \widetilde{\X}$, not necessarily positive, be anti-symmetric with respect to $x_1,\;x_2,\;\cdots,\; x_m$, and suppose that
\begin{equation}
\label{nonboundf}
0 < \liminf_{|x|\to 0, \, x\in \Omega}\frac{|x|^{\gamma + 2m}}{ x_1\cdots x_m} f(x)\leq \limsup_{|x|\to 0, \, x\in \Omega}\frac{|x|^{\gamma + 2m}}{ x_1\cdots x_m}f(x) < \infty.
\end{equation}
Then there exists $\lambda_0>0$ such that for all $\;\lambda>\lambda_0,$ the maximal solution $u_\lambda: (0,T_{\max}(\lambda f))\to C_0(\R^N)$ of the integral equation \eqref{NLHint}
with initial value $\lambda f$ blows up in finite time. Moreover,
\begin{equation}
\label{nonlimsupliminf}0 < \liminf_{\lambda\nearrow \infty}\lambda^{[({1\over \alpha}-{\gamma+m\over 2})^{-1}]}T_{\max}(\lambda f)\leq \limsup_{\lambda\nearrow \infty}\lambda^{[({1\over \alpha}-{\gamma+m\over 2})^{-1}]}T_{\max}(\lambda f) < \infty.
\end{equation}
\end{prop}

\begin{proof} The proof follows exactly the same steps as the proof of
Theorem~\ref{Dickstein-type} with the obvious modifications.  For example, the
first inequality in the proof is true for $|x| \le L$, instead of $|x| \ge L$, etc.

\end{proof}

\begin{prop}
\label{nonnonexistlimit} Let $a=1.$ Let the  integer $m$ and the real numbers $\alpha,\; \gamma$ be such that
$$0\leq m\leq N,\;  0<\gamma <N,\; 0<\alpha<{2\over \gamma+m},\; (N-2)\alpha<4.$$
Then there exists $f\in \widetilde{\X}$ satisfying the hypotheses of Proposition \ref{nonDickstein-type} such that
$$
\liminf_{\lambda\to \infty}\lambda^{[({1\over \alpha}-{\gamma+m\over 2})^{-1}]}T_{\max}(\lambda f) < \limsup_{\lambda\to \infty}\lambda^{[({1\over \alpha}-{\gamma+m\over 2})^{-1}]}T_{\max}(\lambda f).
$$
\end{prop}

\begin{proof} The proof is same as that of Corollary~\ref{nonexistlimit}, however
taking into account Remark~2.4~(vi) of \cite{CDW-DCDS}.  More precisely,
 if we set
$$
{\mathcal J}f(x) = |x|^{-2(\gamma + m)}f(x/|x|^2)
$$
for all $f \in \widetilde{\mathcal{B}_K}$,
then ${\mathcal J}$ is a homeomorphism of $\widetilde{\mathcal{B}_K^*}$,
${\mathcal J}$ leaves invariant all functions which are homogeneous of
degree $-(\gamma + m)$,
$$
\lambda^{\gamma+m}D_\lambda{\mathcal J} = (1/\lambda)^{\gamma+m}{\mathcal J}D_{1/\lambda}
$$
and
$$
{\mathcal H} (f) = {\mathcal J}^{-1}\mathcal{Z}({\mathcal J}(f)),
$$
 where
\begin{equation}
\label{Hu0}
{\mathcal{H}}(f):=\left\{z\in {\mathcal{B}_K},\; \exists\, \lambda_n\to 0, n\to \infty,\; \mbox{ such that }\;
\lim_{n\to \infty}\lambda_n^{\gamma+m}D_{\lambda_n}f=z\, \mbox{in}\, {\mathcal{B}}_K^*\right\}.
\end{equation}
Thus, there exists $f \in {\mathcal B}_K^*$ such that
${\mathcal{H}}(f) = \{c\psi_0 ;
c_1 \le c \le c_2\}$.
\end{proof}

\begin{rems}$\;${\rm
\begin{itemize}
\item[(i)]
The functions $f \in {\widetilde\X}$ satisfying the hypotheses of Propostion~\ref{nonDickstein-type}, and in particular $f \in {\widetilde\X}$ such that
${\mathcal{H}}(f) = \{c\psi_0 ;
c_1 \le c \le c_2\}$, contructed in the previous proof, clearly can {\it not} be bounded in
any neighborhood of $0$.
\item[(ii)] On the other hand, if $\gamma + m < N$, there exist $f \in {\widetilde\X}$ satisfying the hypotheses of Propostion~\ref{nonDickstein-type} with
$f \in L^q(\Omega)$, $q > 1$ and $q_c < q < \frac{N}{\gamma + m}$,
where $q_c = \frac{N\alpha}{2}$.  In other words, the ``classical" theory
of local wellposedness for \eqref{NLheat} in $L^q$ includes such initial values.
\end{itemize}
}
\end{rems}

We conclude with  the proof of Theorem \ref{lifespanderiveedelta}.
\begin{proof}[Proof of Theorem~\ref{lifespanderiveedelta}]
This result is just a small addition to \cite[Theorem 1.2, p. 508]{TW-I}
and the proof is almost completely given there.
To complete the proof, it suffice to observe, {\it using the notation
established in  \cite[pp. 521-523]{TW-I}},
 that $T_{\max} (v_\tau(t_0)) \to T_{\max}(u(t_0))$
as $\tau \to 0$.  This is true by continuity of the maximal time of existence in $C_0(\R^N)$ for type I blow-up solutions  (see \cite{FMZ}).
This implies,  {\it continuing with the notation
established in  \cite[pp. 521-523]{TW-I}}, that
\begin{equation*}
\lim_{\lambda\to 0}\lambda^{[({1\over \alpha}-{N+m\over 2})^{-1}]}T_{\max}(\lambda f) = \lim_{\tau\to 0}T_{\max}(f_\tau)
= T_{\max} (v_\tau(t_0)) + t_0 = T_{\max}(u(t_0)) + t_0
= T(u_0).
\end{equation*}
 This finishes the proof of the theorem.
\end{proof}

\end{document}